\documentclass{birkjour}
\usepackage{amssymb}
\usepackage{amsthm}
\usepackage{mathrsfs}
\usepackage{mathtools}
\usepackage{esint}
 \newtheorem{thm}{Theorem}[section]
 \newtheorem{cor}[thm]{Corollary}
 \newtheorem{lem}[thm]{Lemma}
 \newtheorem{prop}[thm]{Proposition}
 \theoremstyle{definition}
 \newtheorem{defn}[thm]{Definition}
 \theoremstyle{remark}
 \newtheorem{rem}[thm]{Remark}
 \newtheorem*{ex}{Example}
 \numberwithin{equation}{section}
 
 \usepackage{lineno}

\begin{document}

%
%
%
%
%
%
%
%
%

\title[\textit{\tiny{Reiterated Periodic Homogenization of  Parabolic Monotone Operators with Nonstandard Growth}}]
 {Reiterated Periodic Homogenization of  Parabolic Monotone Operators with Nonstandard Growth}

\author[\tiny{\textsc{Franck Tchinda}}]{\textsc{Franck Tchinda}}

\address{%
University of Maroua\\
Department of Mathematics and Computer Science\\
P.O. Box 814\\
Maroua, Cameroon}

\email{takougoumfranckarnold@gmail.com}

\author[\tiny{\textsc{Joel Fotso Tachago}}]{\textsc{Joel Fotso Tachago}}

\address{%
	University of Bamenda\\
	Higher Teachers Trainning College, Department of Mathematics\\
	P.O. Box 39\\
	Bambili, Cameroon}

\email{fotsotachago@yahoo.fr}

\author[\tiny{\textsc{Joseph Dongho}}]{\textsc{Joseph Dongho}}
\address{%
	University of Maroua\\
	Department of Mathematics and Computer Science\\
	P.O. Box 814\\
	Maroua, Cameroon}
\email{joseph.dongho@fs.univ-maroua.cm}


\date{\today}

\begin{abstract}
In this paper, we are interested in reiterated periodic homogenization for a family of parabolic problems with nonstandard growth monotone operators leading to Orlicz spaces. The aim of this work is the determination of the global homogenized problem on the one hand and the macroscopic homogenized problem on the other hand, via the reiterated two-scale convergence method adapted to this type of spaces.
\end{abstract}

\maketitle


\section{Introduction} \label{sect1} 

Let $B : [0,\infty) \to [0,\infty)$ be a differentiable $N$-function and let $\widetilde{B}$ be the Fenchel's conjugate of $B$.   
 Let $\Omega$ be a smooth bounded open set in $\mathbb{R}^{d}_{x}$ (the space of variables $x= (x_{1}, \cdots, x_{d})$ (integer $d\geq1$)). Let $T$ be a positive real number. Let $f\in L^{\widetilde{B}}(0,T ; W^{-1}L^{\widetilde{B}}(\Omega,\mathbb{R}))\equiv \left(L^{B}(0,T ; W^{1}_{0}L^{B}(\Omega,\mathbb{R})) \right)^{\prime}$. For each given $\varepsilon>0$, we consider the initial-boundary value problem 
\begin{equation}\label{c1eq2}
\left\{\begin{array}{l}
\dfrac{\partial u_{\varepsilon}}{\partial t} - \textup{div}\, a\left(\dfrac{x}{\varepsilon},\dfrac{t}{\varepsilon},\dfrac{x}{\varepsilon^{2}}, Du_{\varepsilon}\right) = f \quad \textup{in}\; Q,  \\

\\
u_{\varepsilon} =0 \quad \textup{on} \; \partial\Omega\times (0,T), \\

\\
u_{\varepsilon}(x,0)=0 \quad \textup{in}\; \Omega, 
\end{array} \right.
\end{equation}
where $Q=\Omega\times (0,T)$, $D$ and div denote the usual gradient and divergence operators in $\Omega$, respectively; $L^{B}(Q)$ and $W^{1}_{0}L^{B}(\Omega,\mathbb{R})$ are Orlicz and Orlicz-Sobolev spaces (respectively) which will be specified later, and where $a (\equiv a_{i}, i=1,...,d)$ is a function from $\mathbb{R}^{d}\times\mathbb{R}\times \mathbb{R}^{d}\times\mathbb{R}^{d}$ into $\mathbb{R}^{d}$. We assume that the function $a$ and the $N$-function $B$ satisfy the following properties : 
\begin{equation}\label{c1eq3}
 \left\{\begin{array}{l}
\textup{for\, each}\, \lambda\in\mathbb{R}^{d}\, \textup{the\, function}\, (y,\tau,z)\to a(y,\tau,z,\lambda), \\

\\
\textup{denoted\, by}\, a(\cdot,\cdot,\cdot,\lambda),\, \textup{is\, measurable\ from}\, \mathbb{R}^{d}\times\mathbb{R}\times\mathbb{R}^{d}\, \textup{into}\, \mathbb{R}^{d}
\end{array}\right.
\end{equation}
\begin{equation}\label{c1eq4}
 \left\{\begin{array}{l}
a(y,\tau,z,\omega)=\omega\, \textup{almost\, everywhere\, (a.e.)\, in}\, (y,\tau,z)\in \mathbb{R}^{d}_{y}\times\mathbb{R}\times\mathbb{R}^{d}_{z}, \\

\\
 \textup{where}\, \omega\, \textup{is\, the\, origin\, in}\, \mathbb{R}^{d} 
\end{array}\right.
\end{equation} 
\begin{equation}\label{c1eq5}
 \left\{\begin{array}{l}
\textup{there\, exist\, the\, constants}\, c_{0}, c_{1}, c_{2} > 0\, \textup{such\, that}\, \textup{a.e.\, in}\,\mathbb{R}^{d}_{y}\times\mathbb{R}\times\mathbb{R}^{d}_{z} : \\

\\
(i) \quad |a(y,\tau,z,\lambda)-a(y,\tau,z,\lambda^{\prime})| \leq c_{0} \widetilde{B}^{-1} \left[B(c_{1}|\lambda-\lambda^{\prime}|)\right]  \\

\\
(ii) \quad \left[a(y,\tau,z,\lambda)-a(y,\tau,z,\lambda^{\prime})\right]\cdot(\lambda-\lambda^{\prime}) \geq c_{2} B(|\lambda-\lambda^{\prime}|), \\

\\
\textup{for\, all}\, \lambda, \lambda^{\prime} \in \mathbb{R}^{d},\, \textup{where\, the}\, dot\, \textup{denotes\, the\, usual\, Euclidean\, inner\, product\,}\\
\textup{in}\, \mathbb{R}^{d}. 
\end{array}\right.
\end{equation}


\begin{equation}\label{c1eq1}
\left\{\begin{array}{l}
\widetilde{B} \in \Delta', \;\; t^{2}\leq B(\rho_{0}t) \\

\\
1 < \rho_{1} \leq \dfrac{t b(t)}{B(t)} \leq \rho_{2} \;\, \textup{for\,any}\; t>0,  \\

\\
\textup{where}\, \rho_{0}>0, \rho_{1}, \rho_{2}\, \textup{ are\, constants,\, and}\\
b\, \textup{ is\, odd,\, increasing\, homeomorphism\, from}\, \mathbb{R}\, \textup{ onto}\, \mathbb{R}\, \textup{ such\, that}\, B(t)= \int_{0}^{t}b(s)ds. 
\end{array}\right.
\end{equation}
\begin{equation}\label{c1eq6}
\left\{\begin{array}{l}
(i) \quad \textup{for\, each}\, \lambda\in \mathbb{R}^{d}\,  a(\cdot,\cdot, \cdot, \lambda),\, \textup{is\, periodic\,that\,is\, (i.e.)\, a.e.\, in}\, \mathbb{R}^{d}_{y}\times\mathbb{R}\times\mathbb{R}^{d}_{z} :  \\

\\
\qquad a(y+k,\tau+m,z+l,\lambda) = a(y,\tau,z,\lambda),\, \textup{for\, all}\, (k,m,l) \in \mathbb{Z}^{d}\times\mathbb{Z}\times\mathbb{Z}^{d} \\

\\
 (ii) \quad \textup{The\, functions}\, a \,\textup{satisfy\, a\, local\, continuity\, assumption\, in\, the\, first\, variable, } \\
\qquad \textup{i.e.\, for\, each\, bounded\, set}\; \Lambda\, \textup{in}\, \mathbb{R}^{d}\, \textup{and}\, \eta>0,\, \textup{there\, exists}\, \rho>0\, \textup{such\, that} \\
 
 \\
 \qquad \textup{if}\, |\xi|\leq \rho,\, \textup{then} \, |a(y-\xi, \tau,z,\lambda)-a(y, \tau,z,\lambda)| \leq \eta,\\
 
 \\
  \textup{for\, all}\, (\tau,z,\lambda)\in \mathbb{R}\times\mathbb{R}^{d}\times\mathbb{R}^{d}\, \textup{and\, for\, almost\, all}\, y\in\Lambda. 
\end{array}\right.
\end{equation}
Provided the differential operator $u\to \textup{div}\,a\left(\dfrac{x}{\varepsilon},\dfrac{t}{\varepsilon},\dfrac{x}{\varepsilon^{2}}, Du\right)$ is rigorously defined, and an existence and uniqueness result for (\ref{c1eq2}) is sketched, we are interested in this paper to the reiterated homogenization of (\ref{c1eq2}), i.e., the limiting behaviour, as $\varepsilon\to 0$, of $u_{\varepsilon}$ (the solution of (\ref{c1eq2})) under the periodic assumption (\ref{c1eq6}). 
\\ Concerning the aspect of reiterated perodic homogenization in Orlicz setting, we refer to \cite{tacha1}, \cite{tacha3} and \cite{tacha2}. We also refer to \cite{guet} and \cite{wou2} for study of problem (\ref{c1eq2}) in the classical sobolev spaces where the operator $a$ has a standard growth. Note that the periodic homogenization of problem (\ref{c1eq2}) in Orlicz-Sobolev spaces was studied in \cite{kenne1}. Moreover, the periodic and reiterated periodic homogenization of nonlinear elliptic operators in Orlicz-Sobolev's spaces was studied in \cite{nnang1} and \cite{tacha4}. 
The extension to more general spaces of Orlicz-Sobolev is motivated by the fact that the two-scale convergence method introduced by Nguetseng \cite{nguet1} and later developed by Allaire \cite{allair1} have been widely adopted in homogenization of PDEs in classical Sobolev spaces neglecting materials where microstructure cannot be conveniently captured by modeling exclusivety by means of those spaces.
\par  However, the reiterated two-scale convergence reflexive theorem for classical Sobolev spaces proved in \cite{luka1}, was recently extended to Orlicz-Sobolev spaces in \cite{tacha3} as follows : \textit{any bounded sequence $(u_{\varepsilon})_{\varepsilon>0}$ in $W^{1}L^{B}(\Omega)$ ($B$ being a $N$-function of class $\Delta_{2}$) admits a subsequence (still denoted $u_{\varepsilon}$) such that, as $\varepsilon\rightarrow 0$, one has :} 
\begin{equation*}
\left\{\begin{array}{l}
u_{\varepsilon} \rightharpoonup u_{0} \;\, \textup{in}\, W^{1}L^{B}(\Omega)-\textup{weakly},  \\

\\
\int_{\Omega} D_{x_{i}} u_{\varepsilon} \varphi\left(x,\dfrac{x}{\varepsilon},\dfrac{x}{\varepsilon^{2}}\right) dx \rightarrow \iiint_{\Omega\times Y\times Z} \left(D_{x_{i}}u_{0} + D_{y_{i}}u_{1} + D_{z_{i}}u_{2}\right)\varphi(x,y,z) dx dy dz,
\end{array} \right.
\end{equation*} 
\textit{with $1\leq i\leq d$, and for all $\varphi \in L^{\widetilde{B}}(\Omega; \mathcal{C}_{per}(Y\times Z))$, where $u_{0} \in W^{1}_{0}L^{B}(\Omega)$, $u_{1}\in L^{B}_{D_{y}}(\Omega; W^{1}_{\#}L^{B}(Y))$ and $u_{2}\in L^{B}_{D_{z}}(\Omega; L^{1}_{per}(Y; W^{1}_{\#}L^{B}(Z)))$, $Y=Z=(0,1)^{d}$, $D_{x_{i}}$, $D_{y_{i}}$ and $D_{z_{i}}$ denote the distributional derivatives with respect to the variables $x_{i}$, $y_{i}$, $z_{i}$. Furthermore, if $\widetilde{B} \in \Delta^{\prime}$, then $u_{1}\in L^{B}(\Omega; W^{1}_{\#}L^{B}(Y))$ and $u_{2}\in L^{B}(\Omega; L^{B}_{per}(Y; W^{1}_{\#}L^{B}(Z)))$ (see section \ref{sect2} for detailed notations).} 
\\ The quasi-reflexivity of this result on $Q=\Omega\times (0,T)$ is fundamental in the proof of the main results of this work. Precisely, using some notations of functions spaces in section \ref{sect3}, we show the following two propositions.
\begin{prop}
 Suppose that  (\ref{c1eq1})-(\ref{c1eq6}) holds. Then, the sequence of solutions of (\ref{c1eq2}), $(u_{\varepsilon})_{\varepsilon>0}$, satisfies, as $\varepsilon\rightarrow 0$, 
\begin{equation*}
\left\{\begin{array}{l}
u_{\varepsilon} \rightharpoonup u_{0} \;\, \textup{in}\, L^{B}(0,T ; W^{1}_{0}L^{B}(\Omega))-\textup{weakly},  \\

\\
\dfrac{\partial u_{\varepsilon}}{\partial t} \rightharpoonup \dfrac{\partial u_{0}}{\partial t} \;\, \textup{in}\, L^{\widetilde{B}}(0,T ; W^{-1}L^{\widetilde{B}}(\Omega))-\textup{weakly}, \\

\\
\int_{Q} D_{x_{i}} u_{\varepsilon}\, \varphi\left(x,t,\dfrac{x}{\varepsilon},\dfrac{t}{\varepsilon},\dfrac{x}{\varepsilon^{2}}\right) dxdt \rightarrow \iint_{Q\times\Gamma\times Z} \left(D_{x_{i}}u_{0} + D_{y_{i}}u_{1} + D_{z_{i}}u_{2}\right)\varphi(x,t,y,\tau,z) dx dt dy d\tau dz,
\end{array} \right.
\end{equation*} 
for all $\varphi \in L^{\widetilde{B}}(Q; \mathcal{C}_{per}(\Gamma\times Z))$, $1\leq i\leq d$, with $\Gamma= Y\times \varTheta$ and $\varTheta=(0,1)$,  where  $(u_{0}, u_{1}, u_{2}) \in \mathbb{F}^{1,B}_{0}:= \mathcal{W}_{0}(0,T ; W^{1}_{0}L^{B}(\Omega)) \times L^{B}(Q\times \varTheta; W^{1}_{\#}L^{B}(Y;\mathbb{R}))\times L^{B}(Q\times \varTheta; L^{B}_{per}(Y; W^{1}_{\#}L^{B}(Z;\mathbb{R})))$ is the unique solution to the global homogenized problem
\begin{equation}\label{c1lem1}
\left\{\begin{array}{l}
\int_{0}^{T} \left(\dfrac{\partial u_{0}}{\partial t}(t), v_{0}(t) \right)  \\

\\
\quad  + \int_{0}^{1}\iiint_{Q\times Y\times Z} a(y, \tau,z, Du_{0}+Du_{1}+Du_{2})\cdot (Dv_{0}+ Dv_{1}+ Dv_{2}) dxdt dy  d\tau dz  \\

\\
= \int_{0}^{T} \left(f(t), v_{0}(t)\right) dt, \\

\\
\textup{for\, all}\, (v_{0}, v_{1}, v_{2}) \in \mathbb{F}^{1,B}_{0}.
\end{array} \right.
\end{equation} 
\end{prop}
\begin{prop}
For every $\varepsilon>0$, let (\ref{c1eq2}) be such that $a$ and $f$ satisfy (\ref{c1eq1})-(\ref{c1eq6}). Let $u_{0} \in \mathcal{W}_{0}\left(0,T; W^{1}_{0}L^{B}(\Omega;\mathbb{R})\right)$ be the solution defined by means of (\ref{c1lem1}). Then, there is a unique solution to the macroscopic homogenized problem 
\begin{equation*}
\left\{\begin{array}{l}
\dfrac{\partial u_{0}}{\partial t} - \textup{div}\,q(Du_{0}) = f \quad \textup{in}\;Q \\

\\
u_{0}=0 \quad \textup{on}\; \partial\Omega\times (0,T) \\

\\
u_{0}(x,0) = 0 \quad \textup{in} \; \Omega,
\end{array}\right.
\end{equation*}
where $q$ is defined as follows. For $\xi\in \mathbb{R}^{d}$
\begin{equation*}
q(\xi) = \int_{\Gamma} h(y,\tau, \xi + D_{y}\pi_{1}(\xi))\,dyd\tau,
\end{equation*}
where, for a.e. $y\in Y$, and for any $\xi\in\mathbb{R}^{d}$  
\begin{equation*}
h(y,\tau,\xi) := \iint_{\varTheta\times Z} a_{i}(y,z,\tau, \xi + D_{z}\pi_{2}(y,\xi))\,dzd\tau,
\end{equation*} 
where, for a.e. $y\in Y$, and for any $\xi\in\mathbb{R}^{d}$, $\pi_{2}(y,\xi)$, is the unique solution to the following variational cell problem : 
\begin{equation*}
\left\{\begin{array}{l}
\textup{find}\; \pi_{2}(y,\xi) \in L^{B}(\varTheta; W^{1}_{\#}L^{B}(Z)) \; \textup{such\, that} \\

\\
\iint_{\varTheta\times Z} a(y,z,\tau, \xi + D_{z}\pi_{2}(y,\xi))\cdot D_{z}\theta\, dzd\tau = 0 \;\;\, \textup{for\, all}\, \theta\in L^{B}(\varTheta; W^{1}_{\#}L^{B}(Z)), 
\end{array}\right.
\end{equation*}
and $\pi_{1}\in L^{B}(\varTheta; W^{1}_{\#}L^{B}(Y))$ is the unique solution to the  variational problem  
\begin{equation*}
\left\{\begin{array}{l}
\textup{find}\; \pi_{1}(\xi) \in L^{B}(\varTheta; W^{1}_{\#}L^{B}(Y)) \; \textup{such\, that} \\

\\
\int_{\Gamma} h(y,\tau, \xi + D_{y}\pi_{1}(\xi))\cdot D_{y}\theta\, dyd\tau = 0 \;\;\, \textup{for\, all}\, \theta\in L^{B}(\varTheta; W^{1}_{\#}L^{B}(Y)). 
\end{array}\right.
\end{equation*}
\end{prop}
\par The paper is organized as follows : Section \ref{sect2} deals with some preliminary results on Orlicz-Sobolev spaces. In section \ref{sect3}, we sketch out the existence and uniqueness, for each $\varepsilon$, of the solution of (\ref{c1eq2}). In section \ref{sect4}, after extending the main reiterated two-scale convergence result in our setting, we study in section \ref{sect5} the reiterated homogenization of problem (\ref{c1eq2}). Finally, in the Appendix, for the reader's convenience, we recall some trace results and we justify the well-posedness of abstract problem for (\ref{c1eq2}) under our set of assumptions.

\section{Preliminaries }\label{sect2}

In what follows $X$ and $V$ denote a locally compact space and a Banach space, respectively, and $\mathcal{C}(X; V)$ stands for the space of continuous functions from $X$ into $V$, and $\mathcal{C}_{b}(X; V)$ stands for those functions in $\mathcal{C}(X; V)$ that are bounded. The space $\mathcal{C}_{b}(X; V)$ is endowed with the supremum norm 
\begin{equation*}
\|u\|_{\infty} = \sup_{x\in X} \|u(x)\|,
\end{equation*}
where $\|\cdot\|$ denotes the norm in $V$, (in particular, given an open set $A\subset \mathbb{R}^{d}$ by $\mathcal{C}_{b}(A)$, we denote the space of real valued continuous and bounded functions defined in $A$). Likewise the spaces $L^{p}(X; V)$ and $L^{p}_{loc}(X; V)$ ($X$ provided with a positive Randon measure) are denoted by $L^{p}(X)$ and $L^{p}_{loc}(X)$, respectively, when $V=\mathbb{R}$ (we refer to \cite{fonse1} for integration theory). \\ In the sequel, we denote by $Y$ and $Z$ two identical copies of the unit cube $(0, 1)^{d}$.\\ In order to enlighten the space variable under consideration, we will adopt the notation $\mathbb{R}^{d}_{x}$, $\mathbb{R}^{d}_{y}$ or $\mathbb{R}^{d}_{z}$ to indicate where $x$, $y$ or $z$ belong to. \\
The family of open subsets in $\mathbb{R}^{d}_{x}$ will be denoted by $\mathcal{A}(\mathbb{R}^{d}_{x})$. \\
For any subset $E$ of $\mathbb{R}^{m}$, $m\in \mathbb{N}$, by $\overline{E}$, we denote its closure in the relative topology. \\
For every $x\in \mathbb{R}^{d}$, we denote by $[x]$ its integer part, namely, the vector in $\mathbb{Z}^{d}$, which has as components the integer parts of the components of $x$. 
By $\mathcal{L}^{d}$, we denote the Lebesgue measure in $\mathbb{R}^{d}$.  

\subsection{$N$-functions}

Let $\Phi : [0, \infty) \rightarrow  [0, \infty)$ be an $N$-function, i.e. : 
\begin{itemize}
	\item[i)] $\Phi$ is continuous and convex ,
	\item [ii)] $\Phi(t) > 0$ for $t > 0$,
	\item[iii)] $\displaystyle \lim_{t \to 0} \dfrac{\Phi(t)}{t} = 0$ and $\displaystyle \lim_{t \to \infty} \dfrac{\Phi(t)}{t} = \infty$.
\end{itemize}
Equivalently, $\Phi$ is of the form 
\begin{equation*}
\Phi(t) = \int_{0}^{t} \phi(\tau)d\tau,
\end{equation*}
 where $\phi : [0, \infty) \rightarrow  [0, \infty)$ is nondecreasing, right continuous, with $\phi(0)=0$, $\phi(t) >0$ if $t>0$ and $\phi(t) \to \infty$ if $t \to \infty$.\\
Given two $N$-functions $\Phi$ and $\Psi$, we say that $\Phi$ dominates $\Psi$ (denoted by $\Phi \succ \Psi$ or $\Psi \prec \Phi$) near infinity if there are $k>1$ and $t_{0}>0$ such that 
\begin{equation*}
\Psi(t) \leq \Phi(kt), \quad \forall t> t_{0}.
\end{equation*}
With this in mind, it is well known that if $\displaystyle \lim_{t\to +\infty} \frac{\Phi(t)}{\Psi(t)} = +\infty$ then $\Phi$ dominates $\Psi$ near infinity. \\
The Fenchel's conjugate (or complementary $N$-function) of the $N$-function $\Phi$, is an $N$-function denoted by $\widetilde{\Phi}$ and defined by :
\begin{equation*}
\widetilde{\Phi}(s) = \sup_{t\geq 0} \left[ ts - \Phi(t) \right], \qquad s\geq 0,
\end{equation*}
and we have the following Young inequality
\begin{equation*}
ts \leq \Phi(t)+\widetilde{\Phi}(s), \qquad t,s\geq 0.
\end{equation*}
An $N$-function $\Phi$ is said to satisfy the $\nabla_{2}$-condition for large $x$ (resp. for small $x$ or for all $x$), which is written as $\Phi \in \nabla_{2}(\infty)$ (resp. $\Phi \in \nabla_{2}(0)$ or $\Phi \in \nabla_{2}$), if there exist constants $x_{0} > 0$, $c > 2$ such that  
\begin{equation*}
\Phi(x) \leq \dfrac{1}{2c} \Phi(cx),
\end{equation*}
for $x \geq x_{0}$ (resp. for $0 \leq x \leq x_{0}$ or for all $x \geq 0$).	
Similarly,  $\Phi$ is said to satisfy the $\Delta_{2}$-condition for large $t$ (resp. for small $t$ or for all $t$), which is written as $\Phi \in \Delta_{2}(\infty)$ (resp. $\Phi \in \Delta_{2}(0)$ or $\Phi \in \Delta_{2}$), if there exist constants  $t_{0} > 0$, $k > 2$ such that  
\begin{equation*}
\Phi(2t) \leq k \Phi(t),
\end{equation*}
for $t \geq t_{0}$ (resp. for $0 \leq t \leq t_{0}$ or for all $t \geq 0$).	\\
Let  $\Phi \in \Delta_{2}$ be an $N$-function. Then there are $k> 0$ and $t_{0} \geq 0$ such that 
\begin{equation*} 
\widetilde{\Phi}(\phi(t)) \leq k \, \Phi(t) \quad \textup{for} \, \textup{all} \, t \geq t_{0}.
\end{equation*}	
Let $t\rightarrow \Phi(t) = \int_{0}^{t} \phi(\tau) d\tau$ be an $N$-function, and let $\widetilde{\Phi}$ be the Fenchel's conjugate of $\Phi$. Then one has
\begin{equation*} 
\left\{ \begin{array}{l}
\dfrac{t\,\phi(t)}{\Phi(t)} \geq 1 \quad (\textup{resp.} \, > 1) \; \textup{if} \, \phi \, \textup{is} \, \textup{strictly} \, \textup{increasing}  \\
\widetilde{\Phi}(\phi(t)) \leq t\,\phi(t) \leq \Phi(2t)
\end{array}\right.
\end{equation*}
for all $t >0$. Moreover, an $N$-function $\Phi$ is said to satisfy the $\Delta'$-condition (or $\Phi$ belongs to the class $\Delta'$) denoted by $\Phi \in \Delta'$ if there exists $\beta > 0$ such that  
\begin{equation*}
\Phi(ts) \leq \beta\, \Phi(t)\Phi(s), \quad \forall t,s \geq 0.
\end{equation*}
\begin{ex}
 Let $p > 1$, the function $t \rightarrow \frac{t^{p}}{p}$ is an $N$-function which satisfy $\Delta_{2}$-condition, $\nabla_{2}$-condition, $\Delta'$-condition and its Fenchel's conjugate is the $N$-function $t \rightarrow \frac{t^{q}}{q}$, where $\frac{1}{p} + \frac{1}{q} = 1$. The function $t \rightarrow t^{p}\ln(1+t)$, ($p \geq 1$) is an $N$-function which satisfy $\Delta_{2}$-condition, while the $N$-functions $t \rightarrow t^{\ln t}$ and $t \rightarrow e^{t^{\alpha}} - 1$, ($\alpha >0$) are not of class $\Delta_{2}$.
\end{ex}
\begin{rem}
	It follows by (\ref{c1eq1}) that $N$-functions $B,\, \widetilde{B} \in \Delta_{2}$. \cite[Lemma C.3]{clem1}
\end{rem}

\subsection{Orlicz-Sobolev's spaces}

Let $U$ be a bounded open set in $\mathbb{R}^{d}$ (integer $d\geq 1$), and let $\Phi$ be an $N$-function. The Orlicz-class $\widehat{L}^{\Phi}(U)$ is defined to be the space of all measurable functions $u: U\rightarrow \mathbb{R}$ such that 
\begin{equation*}
\int_{U} \Phi\left(\dfrac{|u(x)|}{\delta} \right)dx < +\infty,
\end{equation*}
for some $\delta=\delta(u) >0$. \\
The Orlicz-space $L^{\Phi}(U)$ is the smallest vector space containing the Orlicz-class $\widehat{L}^{\Phi}(U)$.\\
On the space $L^{\Phi}(U)$ we define two equivalent norms :
\begin{itemize}
	\item[i)] the Luxemburg norm, 
	\begin{equation*}
	\lVert u \rVert_{L^{\Phi}(U)} = \inf \left\{ \delta>0 \; : \; \int_{U}^{} \Phi\left(\dfrac{|u(x)|}{\delta} \right)dx  \leq 1 \right\}, \quad \forall u \in L^{\Phi}(U),
	\end{equation*}
	\item[ii)] the Orcliz norm, 
	\begin{equation*}
	\begin{array}{rcc}
	\lVert u \rVert_{L^{(\Phi)}(U)}& =& \sup \left\{ \left| \int_{U} u(x)v(x) dx \right| \; : \; v \in L^{\widetilde{\Phi}}(U) \; \textup{and} \; \|v\|_{L^{\widetilde{\Phi}}(U)} \leq 1 \right\}, \\
	& &  \forall u \in L^{\Phi}(U),
	\end{array}
	\end{equation*}
\end{itemize}
which makes it a Banach space.  \\
Now we will give some properties of Orlicz spaces and we will refer to \cite{adam,adams} for more details.
Assume that $\Phi \in \Delta_{2}$. Then :
\begin{itemize}
	\item[i)] $\mathcal{D}(U)$ is dense in $L^{\Phi}(U)$
	\item[ii)] $L^{\Phi}(U)$ is separable and reflexive whenever $\widetilde{\Phi}\in \Delta_{2}$,
	\item[iii)] the dual of $L^{\Phi}(U)$ is identified with $L^{\widetilde{\Phi}}(U)$, and the dual norm on $L^{\widetilde{\Phi}}(U)$ is equivalent to $\lVert \cdot \rVert_{L^{\widetilde{\Phi}}(U)}$,
	\item[iv)] given $u \in L^{\Phi}(U)$ and $v \in L^{\widetilde{\Phi}}(U)$ the product $uv$ belongs to $L^{1}(U)$ with the generalized H\"{o}lder's inequality 
	\begin{equation}\label{c2eq19}
	\left| \int_{U} u(x)v(x)dx \right| \leq 2\, \lVert u \rVert_{L^{\Phi}(U)} \, \lVert v \rVert_{L^{\widetilde{\Phi}}(U)},
	\end{equation}
	\item[v)] given $v \in L^{\Phi}(U)$ the linear functional $L_{v}$ on $L^{\widetilde{\Phi}}(U)$ defined by 
	\begin{equation*}
	L_{v}(u) = \int_{U} u(x)v(x) dx, \quad u \in L^{\widetilde{\Phi}}(U),
	\end{equation*} 
	belongs to the dual $[L^{\widetilde{\Phi}}(U)]'$ with $\lVert v \rVert_{L^{(\Phi)}(U)} \leq \lVert L_{v} \rVert_{[L^{\widetilde{\Phi}}(U)]'} \leq 2 \lVert v \rVert_{L^{(\Phi)}(U)}$,
	\item[vi)] $L^{\Phi}(U) \hookrightarrow L^{1}(U) \hookrightarrow	L^{1}_{\textup{loc}}(U) \hookrightarrow \mathcal{D'}(U)$, each embedding being continuous,
	\item[vii)] given two $N$-functions $\Phi$ and $\Psi$, we have the continuous embbeding \\ $L^{\Phi}(U) \hookrightarrow L^{\Psi}(U)$ if and only if $\Phi \succ \Psi$ near infinity,
	\item[viii)] the product space $L^{\Phi}(U)^{d}= L^{\Phi}(U)\times L^{\Phi}(U) \times \cdots \times L^{\Phi}(U)$, ($d$-times), is endowed with the norm 
	\begin{equation*}
	\lVert \textbf{v} \rVert_{L^{(\Phi)}(U)^{d}} = \sum_{i=1}^{d} \lVert v_{i} \rVert_{L^{(\Phi)}(U)}, \quad \textbf{v}=(v_{i}) \in L^{\Phi}(U)^{d}.
	\end{equation*}
	\item[ix)]  If $U_{1} \subset \mathbb{R}^{d_{1}}$ and $U_{2} \subset \mathbb{R}^{d_{2}}$ are two bounded open sets with $d_{1}+d_{2}=d$, and if $u \in L^{\Phi}(U_{1}\times U_{2})$,
	then for almost all $x_{1}\in U_{1}$, $u(x_{1}, \cdot) \in L^{\Phi}(U_{2})$. If in addition $\widetilde{\Phi} \in \Delta'$ associate with a constant $\beta$, then the function $u$ belongs to $L^{\Phi}(U_{1}, L^{\Phi}(U_{2}))$, with	
	\begin{equation}\label{be1}
	\|u\|_{L^{\Phi}(U_{1}, L^{\Phi}(U_{2}))} \leq \iint_{U_{1}\times U_{2}} \Phi(|u(x_{1},x_{2})|) dx_{1}dx_{2} + \beta.
	\end{equation}
\end{itemize}
Moreover, we will be interested with the properties below :
\begin{lem}\label{c2lem2}
	Let $\Phi$ be the $N$-function of (\ref{c1eq1}). If $v\in L^{\Phi}(U)$ then \cite{miha1} : 
	\begin{itemize}
		\item[i)] $\|v\|_{\Phi,U} > 1$ implies $\|v\|_{\Phi,U}^{\rho_{1}} \leq \int_{U} \Phi(|v(x)|) dx \leq \|v\|_{\Phi,U}^{\rho_{2}}$
		\item[ii)] $\|v\|_{\Phi,U} < 1$ implies $\|v\|_{\Phi,U}^{\rho_{2}} \leq \int_{U} \Phi(|v(x)|) dx \leq \|v\|_{\Phi,U}^{\rho_{1}}$
	\end{itemize}
and if $v\in L^{\widetilde{\Phi}}(U)$ then \cite[Lemma C.7]{clem1} :
\begin{itemize}
	\item[iii)] $\|v\|_{\widetilde{\Phi},U} < 1$ implies $ \int_{U} \widetilde{\Phi}(|v(x)|) dx \leq \|v\|_{\widetilde{\Phi},U}^{\frac{\rho_{1}}{\rho_{1}-1}}$.
\end{itemize}
\end{lem}
Analogously, one can define the Orlicz-Sobolev function space as follows : 
\begin{equation*}
W^{1}L^{\Phi}(U) = \left\{ u\in L^{\Phi}(U) \,: \, \dfrac{\partial u}{\partial x_{i}} \in L^{\Phi}(U),\, 1 \leq i \leq d  \right\},
\end{equation*}
where derivatives are taken in the distributional sense on $U$. Endowed with the norm 
\begin{equation*}
\|u\|_{W^{1}L^{\Phi}(U)} = \|u\|_{L^{\Phi}(U)} + \sum_{i=1}^{d} \left\|\dfrac{\partial u}{\partial x_{i}}\right\|_{L^{\Phi}(U)}, \quad u \in W^{1}L^{\Phi}(U), 
\end{equation*}
$W^{1}L^{\Phi}(U)$ is a reflexive Banach space. If $\Phi \in \Delta_{2}$ and the boundary of $U$ is lipschitzian, then the canonical embedding $W^{1}L^{\Phi}(U) \subset L^{\Phi}(U)$ is compact \cite{adam}.
 We denote by $W^{1}_{0}L^{\Phi}(U)$, the closure of $\mathcal{D}(U)$ in $W^{1}L^{\Phi}(U)$ and the semi-norm 
\begin{equation*}
u \rightarrow \|u\|_{W^{1}_{0}L^{\Phi}(U)} = \|Du\|_{\Phi,U} = \sum_{i=1}^{d} \left\|\dfrac{\partial u}{\partial x_{i}}\right\|_{L^{\Phi}(U)}
\end{equation*}
is a norm on $W^{1}_{0}L^{\Phi}(U)$ equivalent to $\|\cdot\|_{W^{1}L^{\Phi}(U)}$. \\
By $W^{1}_{\#}L^{\Phi}(Y)$, we denote the space of functions $u \in W^{1}L^{\Phi}(U)$ such that 
\begin{equation*}
\int_{Y} u(y) dy = 0.
\end{equation*}
It is endowed with the gradient norm. \\
Given a function space $S$ defined in $Y$, $Z$ or $Y\times Z$, the subscript $_{per}$ stands for periodic, i.e., $S_{per}$ means that its elements are periodic in $Y$, $Z$ or $Y\times Z$, as it will be clear from the context. In particular, $\mathcal{C}_{per}(Y\times Z)$ denotes the space of periodic functions in $\mathcal{C}(\mathbb{R}^{d}_{y}\times \mathbb{R}^{d}_{z})$, i.e., 
\begin{equation*}
w(y+k, z+h) = w(y, z) \quad \textup{for} \, (y, z) \in \mathbb{R}^{d}\times \mathbb{R}^{d}\; \textup{and} \; (k, h) \in \mathbb{Z}^{d}\times \mathbb{Z}^{d}. 
\end{equation*}
\begin{equation*}
\mathcal{C}^{\infty}_{per}(Y\times Z) = \mathcal{C}_{per}(Y\times Z) \cap \mathcal{C}^{\infty}(\mathbb{R}^{d}_{y}\times \mathbb{R}^{d}_{z}),
\end{equation*}
and $L^{\Phi}_{per}(Y\times Z)$ is the space of $Y\times Z$-periodic functions in $L^{\Phi}_{loc}(\mathbb{R}^{d}_{y}\times \mathbb{R}^{d}_{z})$. In our subsequent analysis, we will denote by $L^{\Phi}(U; L^{\Phi}_{per}(Y))$ and $L^{\Phi}(U; L^{\Phi}_{per}(Y\times Z))$ the spaces of functions in $L^{\Phi}_{loc}(U\times Y)$ and $L^{\Phi}_{loc}(U\times \mathbb{R}^{d}_{y}\times\mathbb{R}^{d}_{z})$, respectively which are $Y$ and $Y\times Z$ periodic for a.e. $x\in U$, respectively and whose Luxemburg norm is finite in $U\times K$, with $K$ being any compact set in $Y$ and $Y\times Z$, respectively, in formulas 
\begin{equation*}
\begin{array}{l}
L^{\Phi}(U; L^{\Phi}_{per}(Y)) := \bigg\{ u \in L^{\Phi}_{loc}(U\times \mathbb{R}^{d}_{y}) : \, u(x, \cdot) \in L^{\Phi}_{per}(Y) \;\,
\textup{for\, a.e.}\, x\in U, \; \textup{and} \\
\qquad\qquad\qquad \qquad\qquad\qquad \iint_{U\times Y} \Phi(|u(x,y)|) dx dy < \infty  \bigg\},
\end{array}
\end{equation*}
\begin{equation*}
\begin{array}{l}
L^{\Phi}(U; L^{\Phi}_{per}(Y\times Z)) := \bigg\{ u \in L^{\Phi}_{loc}(U\times \mathbb{R}^{d}_{y}\times \mathbb{R}^{d}_{z}) : \, u(x, \cdot, \cdot) \in L^{\Phi}_{per}(Y\times Z) \;\,
\textup{for\, a.e.}\, x\in U, \; \textup{and} \\
 \qquad\qquad\qquad \qquad\qquad\qquad \iiint_{U\times Y\times Z} \Phi(|u(x,y,z)|) dx dy dz < \infty  \bigg\},
\end{array}
\end{equation*}
respectively. Furthrmore, we set
\begin{equation*}
L^{\Phi}_{D_{y}}(U; W^{1}_{\#}L^{\Phi}(Y)) := \left\{ u \in L^{1}(U; W^{1}_{\#}L^{\Phi}(Y)) \, :\, D_{y}u \in L^{\Phi}_{per}(U\times Y)^{d}  \right\},
\end{equation*}
\begin{equation*}
L^{\Phi}_{D_{z}}(U; L^{1}_{per}(Y;W^{1}_{\#}L^{\Phi}(Z))) := \left\{ u \in L^{1}(U;L^{1}_{per}(Y; W^{1}_{\#}L^{\Phi}(Z))) \, :\, D_{z}u \in L^{\Phi}_{per}(U\times Y\times Z)^{d}  \right\}.
\end{equation*}
 These spaces play an important role in the definition of reiterated two-scale convergence in the Orlicz setting.
In the next section, we are interested to the existence and uniqueness result for problem (\ref{c1eq2}).

\section{Existence and uniqueness result for (\ref{c1eq2})}\label{sect3}

This aim of this section consists in showing the existence and uniqueness of the solution to (\ref{c1eq2}), neglecting the periodicity assumption on $a$.
\par Let $\textbf{v}= (v_{i}) \in \mathcal{C}(\overline{Q};\mathbb{R})^{d} = \mathcal{C}(\overline{Q};\mathbb{R}) \times \cdots \times \mathcal{C}(\overline{Q};\mathbb{R})$ ($d$ times). One can check, using assumptions (\ref{c1eq3})-(\ref{c1eq1}), that the function $(x,t,y,\tau,z) \rightarrow a(y,\tau,z, \textbf{v}(x,t))$ belongs to $\mathcal{C}(\overline{Q}; L^{\infty}(\mathbb{R}^{d+1}_{y,\tau}; \mathcal{C}_{per}(Z)))^{d}$. Hence for each $\varepsilon>0$, the function $x\rightarrow a\left(\frac{x}{\varepsilon},\frac{t}{\varepsilon},\frac{x}{\varepsilon^{2}},\textbf{v}(x,t)\right)$ of $Q$ into $\mathbb{R}^{d}$, denoted by $a^{\varepsilon}(-,\cdot,\textbf{v})$, is well defined as an element of $L^{\infty}(Q;\mathbb{R})^{d}$ (see \cite{wou2}), and we have the following proposition and corollary.
\begin{prop}\label{c3eq2}
Let $a$ be the function in (\ref{c1eq2}). Assume that hypotheses (\ref{c1eq3})-(\ref{c1eq6}) hold. Then for each $\varepsilon>0$, the transformation $\textbf{v} \rightarrow a^{\varepsilon}(-,\cdot,\textbf{v})$ of $\mathcal{C}(\overline{Q};\mathbb{R})^{d}$ into $L^{\infty}(Q;\mathbb{R})^{d}$ extends by continuity to a mapping, still denoted by $\textbf{v} \rightarrow a^{\varepsilon}(-,\cdot,\textbf{v})$, of $L^{B}(Q;\mathbb{R})^{d}$	into $L^{\widetilde{B}}(Q;\mathbb{R})^{d}$. Moreover, for all $\textbf{v}, \textbf{w} \in L^{B}(Q;\mathbb{R})^{d}$, we have 
\begin{equation}\label{c3eq1}
\|a^{\varepsilon}(-,\cdot,\textbf{v}) - a^{\varepsilon}(-,\cdot,\textbf{w})\|_{L^{\widetilde{B}}(Q)^{d}} \leq c \|\textbf{v} - \textbf{w}\|^{\rho}_{L^{B}(Q)^{d}},
\end{equation}
where $\rho \in \{\rho_{1}, \rho_{2}, \rho_{1}-1, \frac{\rho_{2}}{\rho_{1}}(\rho_{1}-1) \}$, and $c= c(c_{0}, c_{1}, \rho) >0$.
\end{prop}
\begin{proof}
	Arguing as in the proof of \cite[Proposition 1]{kenne1}. Putting 
	\begin{equation*}
	u(x,t,y,\tau,z) = c_{0} \widetilde{B}^{-1}(B(c_{1}|\textbf{v}(x,t)-\textbf{w}(x,t|)) - |a(y,\tau,z,\textbf{v}(x,t)) - a(y,\tau,z,\textbf{w}(x,t))|,
	\end{equation*}
	$\textbf{v}, \textbf{w} \in \mathcal{C}(\overline{Q};\mathbb{R}^{d})$, and for each $\varepsilon>0$, it follows by applying Lemma 2.2 in \cite{gabri2} that 
	\begin{equation*}
	|a^{\varepsilon}(-,\cdot,\textbf{v}) - a^{\varepsilon}(-,\cdot,\textbf{w})| \leq c_{0} \widetilde{B}^{-1}(B(c_{1}|\textbf{v}-\textbf{w}|)) \quad \textup{in} \;\; Q.
	\end{equation*}
	Since $\widetilde{B}$ is increasing, for $\delta>0$ we have 
	\begin{equation*}
	\int_{Q} \widetilde{B}\left(\dfrac{|a^{\varepsilon}(-,\cdot,\textbf{v}) - a^{\varepsilon}(-,\cdot,\textbf{w})|}{\delta}\right) dxdt \leq \int_{Q} \widetilde{B}\left(\dfrac{c_{0} \widetilde{B}^{-1}(B(c_{1}|\textbf{v}-\textbf{w}|))}{\delta}\right) dxdt.
	\end{equation*}
	If $\|a^{\varepsilon}(-,\cdot,\textbf{v}) - a^{\varepsilon}(-,\cdot,\textbf{w})\|_{\widetilde{B},Q}<1$, then the point (iii) of Lemma \ref{c2lem2} combined with the above inequality when $\delta=1$ and the fact that $\widetilde{B}\in \Delta'$ (with constant $\eta$) imply
	\begin{equation*}
	\begin{array}{rcl}
	 \|a^{\varepsilon}(-,\cdot,\textbf{v}) - a^{\varepsilon}(-,\cdot,\textbf{w})\|_{\widetilde{B},Q}^{\frac{\rho_{1}}{\rho_{1}-1}} & \leq & \int_{Q} \widetilde{B}\left(c_{0} \widetilde{B}^{-1}(B(c_{1}|\textbf{v}-\textbf{w}|))\right) dxdt \\
	  & \leq & \eta \widetilde{B}(c_{0}) \int_{Q} B\left(B(c_{1}|\textbf{v}-\textbf{w}|)\right) dxdt.
	\end{array}
	\end{equation*}
	Applying points (i)-(ii) of Lemma \ref{c2lem2} yields (with $\rho \in \{\rho_{1}-1, \frac{\rho_{2}}{\rho_{1}}(\rho_{1}-1)\}$) 
	\begin{equation*}
	\|a^{\varepsilon}(-,\cdot,\textbf{v}) - a^{\varepsilon}(-,\cdot,\textbf{w})\|_{\widetilde{B},Q} \leq \left[\eta \widetilde{B}(c_{0})\right]^{\frac{\rho_{1}-1}{\rho_{1}}} c_{1}^{\rho} \|\textbf{v}-\textbf{w}\|^{\rho}_{B,Q}.
	\end{equation*}
	Let us assume now that $\|a^{\varepsilon}(-,\cdot,\textbf{v}) - a^{\varepsilon}(-,\cdot,\textbf{w})\|_{\widetilde{B},Q} >1$. Since 
	\begin{equation*}
	\left\{\delta>0\, : \, \int_{Q} \widetilde{B}\left(\dfrac{c_{0} \widetilde{B}^{-1}(B(c_{1}|\textbf{v}-\textbf{w}|))}{\delta}\right) dxdt \leq 1  \right\}
	\end{equation*}
	is a subset of $\left\{\delta>0\, : \, \int_{Q} \widetilde{B}\left(\dfrac{|a^{\varepsilon}(-,\cdot,\textbf{v}) - a^{\varepsilon}(-,\cdot,\textbf{w})|}{\delta}\right) dxdt \leq 1  \right\}$, we get 
	\begin{equation*}
	\|a^{\varepsilon}(-,\cdot,\textbf{v}) - a^{\varepsilon}(-,\cdot,\textbf{w})\|_{\widetilde{B},Q} \leq c_{0} \|\widetilde{B}^{-1}(B(c_{1}|\textbf{v}-\textbf{w}|))\|_{\widetilde{B},Q}.
	\end{equation*}
	If $\|\widetilde{B}^{-1}(B(c_{1}|\textbf{v}-\textbf{w}|))\|_{\widetilde{B},Q} \leq 1$ then 
	\begin{equation*}
	\|\widetilde{B}^{-1}(B(c_{1}|\textbf{v}-\textbf{w}|))\|_{\widetilde{B},Q} \leq \left(\int_{Q} B(c_{1}|\textbf{v}-\textbf{w}|) dxdt\right)^{\frac{\rho_{1}-1}{\rho_{1}}} \leq c_{1}^{\rho} \|\textbf{v}-\textbf{w}\|^{\rho}_{B,Q},
	\end{equation*}
	where  $\rho \in \{ \rho_{1}-1, \frac{\rho_{2}}{\rho_{1}}(\rho_{1}-1) \}$. Suppose now $	\|\widetilde{B}^{-1}(B(c_{1}|\textbf{v}-\textbf{w}|))\|_{\widetilde{B},Q} >1$. Hence there is an integer $P_{0}$ such that $	\|\widetilde{B}^{-1}(B(c_{1}|\textbf{v}-\textbf{w}|))\|_{\widetilde{B},Q} - \frac{1}{n} >1$ for all integers $n\geq P_{0}$ and 
	\begin{equation*}
	1 < \int_{Q} \widetilde{B}\left(\left(\|\widetilde{B}^{-1}(B(c_{1}|\textbf{v}-\textbf{w}|))\|_{\widetilde{B},Q} - \frac{1}{n}\right)^{-1} \widetilde{B}^{-1}(B(c_{1}|\textbf{v}-\textbf{w}|)) \right) dxdt.
	\end{equation*}
	Using the convexity of $\widetilde{B}$ yields 
	\begin{equation*}
	\|\widetilde{B}^{-1}(B(c_{1}|\textbf{v}-\textbf{w}|))\|_{\widetilde{B},Q} - \frac{1}{n} < \int_{Q} B(c_{1}|\textbf{v}-\textbf{w}|) dxdt \quad \textup{for \, all}\; n\geq P_{0}.
	\end{equation*}
	Passing to the limit (as $n\to \infty$) and taking into account (i)-(ii) of Lemma \ref{c2lem2}, we are led to 
	\begin{equation*}
	\|\widetilde{B}^{-1}(B(c_{1}|\textbf{v}-\textbf{w}|))\|_{\widetilde{B},Q} \leq c_{1}^{\rho} \|\textbf{v}-\textbf{w}\|_{B,Q}^{\rho}, \quad \textup{with}\; \rho=\rho_{1}\, \textup{or} \, \rho_{2}.
	\end{equation*}
	Therefore (\ref{c3eq1}) follows for all $\textbf{v}, \textbf{w} \in \mathcal{C}(\overline{Q};\mathbb{R})^{d}$, then for all $\textbf{v}, \textbf{w} \in L^{B}(\overline{Q};\mathbb{R})^{d}$ by density and continuity.
\end{proof}

\begin{cor}\label{c3cor1}
	Let hypotheses be (\ref{c1eq3})-(\ref{c1eq6}). Given $v \in L^{B}(0,T ; W^{1}L^{B}(\Omega;\mathbb{R}))$, the function $(x,t)\rightarrow a\left(\frac{x}{\varepsilon}, \frac{t}{\varepsilon},\frac{x}{\varepsilon^{2}}, Dv(x,t)\right)$ from $Q$ into $\mathbb{R}^{d}$, denoted by $a^{\varepsilon}(-,\cdot,Dv)$, is well defined as an element $L^{\widetilde{B}}(Q;\mathbb{R})^{d}$. Moreover we have 
	\begin{equation}
	\begin{array}{l}
	a^{\varepsilon}(-,\cdot, \omega)=\omega \quad \textup{in}\, Q\, (\omega\, \textup{being the origin in} \in \mathbb{R}^{d}), \\
	
	\\
	\left[a^{\varepsilon}(-,\cdot, Dv) - a^{\varepsilon}(-,\cdot, Dw) \right]\cdot (Dv - Dw) \geq c_{2} B(|Dv - Dw|)\; \textup{in}\, Q 
	\end{array}
	\end{equation}
	for all $v, w \in L^{B}(0,T ; W^{1}L^{B}(\Omega;\mathbb{R}))$. 
\end{cor}
\par Let $\theta \in \mathcal{D}(\mathbb{R})$ with $0\leq \theta \leq 1$, $\int \theta(t)dt=1$, $\theta$ supports in the closed interval $[-1,1]$. For each integer $n\geq 1$, we put $\theta_{n}(t) = n \theta(nt)$, $t\in\mathbb{R}$, so that the sequence $(\theta_{n})_{n\geq 1}$ is a mollifier on $\mathbb{R}$. For fixed $x\in\Omega$ and $\lambda\in\mathbb{R}^{d}$, let $\bar{a}^{\varepsilon}(x,\cdot,\lambda)$ and $\bar{f}(x,\cdot)$ be functions on $\mathbb{R}$ defined by 
\begin{equation*}
\bar{a}^{\varepsilon}_{i}(x,t,\lambda) = \left\{\begin{array}{l}
a_{i}\left(\frac{x}{\varepsilon},\frac{t}{\varepsilon},\frac{x}{\varepsilon^{2}}, \lambda\right)\;\; \textup{if}\, t\in (0,T) \\
0 \;\; \textup{if}\, t\in \mathbb{R}\backslash (0,T), 
\end{array} \right. \;\; \textup{and} \;\; 
\bar{f}(x,t) = \left\{\begin{array}{l}
f(x,t)\;\; \textup{if}\, t\in (0,T) \\
0 \;\; \textup{if}\, t\in \mathbb{R}\backslash (0,T). 
\end{array} \right.
\end{equation*}
For $n\geq 1$, let $a^{\varepsilon}_{in}(x,\cdot,\lambda)$, $1\leq i \leq d$, and $f_{n}(x,\cdot)$ be the restrictions in $(0,T)$ of functions 
\begin{equation*}
t \rightarrow \int \theta_{n}(\tau)\bar{a}_{i}^{\varepsilon}(x,t-\tau,\lambda) d\tau \;\; \textup{and} \;\; t \rightarrow \int \theta_{n}(\tau)\bar{f}(x,t-\tau) d\tau , \; \textup{(respectively)}.
\end{equation*}
Then $f_{n}(x,\cdot) \in \mathcal{C}([0,T])$ and $a^{\varepsilon}_{n}(x,\cdot,\lambda)= (a^{\varepsilon}_{in}(x,\cdot,\lambda)) \in \mathcal{C}([0,T])^{d}$. Next, it is obviuos that 
\begin{equation}\label{c2eq8}
\begin{array}{l}
a^{\varepsilon}_{n}(-,\cdot, \omega)=\omega, \\

\\
|a^{\varepsilon}_{n}(-,\cdot, \lambda) - a^{\varepsilon}_{n}(-,\cdot, \mu)|\leq c_{0} \widetilde{B}^{-1}(B(c_{1}|\lambda-\mu|)), \;\; \textup{a.e.\, in}\, Q  \\

\\
\left(a^{\varepsilon}(-,\cdot, \lambda) - a^{\varepsilon}(-,\cdot, \mu) \right)\cdot (\lambda - \mu) \geq c_{2} B(|\lambda-\mu|) 
\end{array}
\end{equation}
for all $\lambda,\mu \in \mathbb{R}^{d}$. Let us consider the problem 
\begin{equation}\label{c2eq9}
\left\{\begin{array}{l}
\dfrac{\partial u_{\varepsilon n}}{\partial t} - \textup{div}\, a^{\varepsilon}_{n}\left(-,\cdot, Du_{\varepsilon n}\right) = f_{n} \quad \textup{in}\; Q,  \\

\\
u_{\varepsilon n}(x,t) =0 \quad \textup{in} \; \partial\Omega\times (0,T), \\

\\
u_{\varepsilon n}(x,0)=0 \quad \textup{in}\; \Omega. 
\end{array} \right.
\end{equation}
Given $u\in L^{B}(0,T ; W^{1}L^{B}(\Omega,\mathbb{R}))$, the function $a^{\varepsilon}_{n}(-,\cdot, Du)$ belongs to $L^{\widetilde{B}}(Q;\mathbb{R})^{d}$. Next, by (\ref{c2eq8}) the transformation $u \rightarrow a^{\varepsilon}_{n}(-,\cdot, Du)$ from $L^{B}(0,T ; W^{1}L^{B}(\Omega,\mathbb{R}))$ into $L^{\widetilde{B}}(Q;\mathbb{R})^{d}$ is continuous with 
\begin{equation}\label{c2eq10}
\left\{\begin{array}{l}
\|a^{\varepsilon}_{n}(-,\cdot,Du) - a^{\varepsilon}_{n}(-,\cdot,Dv)\|_{L^{\widetilde{B}}(Q;\mathbb{R})^{d}} \leq c \|Du - Dv\|^{\rho}_{L^{B}(Q;\mathbb{R})^{d}},  \\

\\
\left[a^{\varepsilon}_{n}(-,\cdot, Du) - a^{\varepsilon}_{n}(-,\cdot, Dv) \right]\cdot (Du - Dv) \geq c_{2} B(|Du - Dv|)\; \textup{a.e.\, in}\, Q, 
\end{array} \right.
\end{equation}
for all $u, v\in L^{B}(0,T ; W^{1}L^{B}(\Omega,\mathbb{R}))$. Let $A^{\varepsilon}_{n}(t) : W^{1}L^{B}(\Omega,\mathbb{R}) \rightarrow W^{-1}L^{\widetilde{B}}(\Omega,\mathbb{R})$ be the operator defined by $A^{\varepsilon}_{n}(t)u = -\textup{div}a^{\varepsilon}_{n}(-,t, Du)$. Then (\ref{c2eq9}) is equivalent to the problem  
\begin{equation}\label{c2eq11}
\left\{\begin{array}{l}
u^{\prime}_{\varepsilon n}(t) + A_{n}^{\varepsilon}(t)u_{\varepsilon n}(t) = f_{n}(t), \quad 0< t\leq T, \\
u_{\varepsilon n}(0) = u_{0}.
\end{array}\right.
\end{equation}
For $V=W^{1}_{0}L^{B}(\Omega,\mathbb{R})$, $H=L^{2}(\Omega;\mathbb{R})$ and $A(t)=A_{n}^{\varepsilon}(t)$ $(0\leq t\leq T)$, the assumption (C) is satisfied (see \textbf{Appendix B}). Thanks to Theorem \ref{apenb1} we deduce the following one : 
\begin{thm}\label{c2eq12}
	For each $n\in \mathbb{N}^{\ast}$ and each $\varepsilon> 0$, there exists one and only one function $u_{\varepsilon n}$ from $[0,T]$ into $W^{1}_{0}L^{B}(\Omega,\mathbb{R})$ defined by (\ref{c2eq11}) and satisfying $u_{\varepsilon n} \in  L^{B}(0,T ; W^{1}L^{B}(\Omega;\mathbb{R})) \cap L^{\infty}(0,T ; L^{2}(\Omega,\mathbb{R}))$.
\end{thm}
\par Let us pass to the limit in (\ref{c2eq11}) when $n\to \infty$. Equation (\ref{c2eq11}) takes the variational formulation 
\begin{equation}\label{c2eq13}
\left(u^{\prime}_{\varepsilon n}(t), v\right) + \int_{\Omega} a^{\varepsilon}_{n}(-,t, Du_{\varepsilon n})\cdot Dv\, dx = \left(f_{n}(t), v\right), 
\end{equation}
with $0< t\leq T$, for all $v\in W^{1}_{0}L^{B}(\Omega;\mathbb{R})$. Letting $v = u_{\varepsilon n}(t)$, $0< t< T$, and arguing as in \cite[Page 220]{kenne1} we have 
\begin{equation}\label{c2eq14}
\|u_{\varepsilon n}(t)\|^{2}_{L^{2}(\Omega)} + c \int_{0}^{T} B\left(\|Du_{\varepsilon n}(t)\|_{L^{B}(\Omega)^{d}}\right) dt \leq \int_{0}^{T} \widetilde{B}\left(\gamma\|f_{n}(t)\|_{W^{-1}L^{\widetilde{B}}(\Omega;\mathbb{R})}\right) dt,
\end{equation} 
with $\gamma= \frac{2}{c}$ if $0<c <1$ and $\gamma=2c$ if $c\geq 1$.\\
Recalling that there is an integer $P\geq 1$ such that $\|f_{n} - f\|_{L^{\widetilde{B}}(0,T ; W^{-1}L^{\widetilde{B}}(\Omega;\mathbb{R}))} \leq 1$ for all integer $n\geq P$; we are derived to 
\begin{equation*}
\sup_{n\in \mathbb{N}}\left(\|u_{\varepsilon n}\|_{L^{\infty}(0,T ; L^{2}(\Omega))}\right) <\infty \;\; \textup{and} \;\; \sup_{n\in \mathbb{N}}\left(\|Du_{\varepsilon n}\|_{L^{B}(0,T ; L^{B}(\Omega))^{d}}\right) <\infty.
\end{equation*}
There exists a subsequence $(u_{\varepsilon k_{n}})_{n\in \mathbb{N}^{\ast}}$ of $(u_{\varepsilon n})_{n\in \mathbb{N}^{\ast}}$ such that, as $n\rightarrow \infty$,
\begin{equation}\label{c2eq15}
\left\{\begin{array}{l}
u_{\varepsilon k_{n}} \rightarrow u_{\varepsilon} \;\; \textup{in} \; L^{\infty}(0,T ; L^{2}(\Omega))-\textup{weak}^{\ast},  \\

\\
\dfrac{\partial u_{\varepsilon k_{n}}}{\partial x_{i}} \rightarrow \dfrac{\partial u_{\varepsilon}}{\partial x_{i}} \;\; \textup{in} \; L^{B}(Q)-\textup{weak} \; (1\leq i\leq d).
\end{array}\right.
\end{equation}
This being so, there is no difficulty in proving that $a^{\varepsilon}_{n}(-,\cdot, Du_{\varepsilon}) \rightarrow a^{\varepsilon}(-,\cdot, Du_{\varepsilon})$ in $L^{\widetilde{B}}(Q)^{d}$ when $n\rightarrow \infty$. Besides, according to (\ref{c2eq10}) and (\ref{c2eq14}), $a^{\varepsilon}(-,\cdot, Du_{\varepsilon k_{n}}) \rightarrow a^{\varepsilon}(-,\cdot, Du_{\varepsilon})$ in $L^{\widetilde{B}}(Q)^{d}-\textup{weak}$ when $n\to \infty$. Therefore, fixing $\varepsilon>0$, $v \in W^{1}_{0}L^{B}(\Omega;\mathbb{R})$ and $\varphi$ in $\mathcal{D}(0,T)$ with $\|v\varphi\|_{L^{B}(0,T ; W^{1}_{0}L^{B}(\Omega))} >0$, there exists $P\in\mathbb{N}$ such that 
\begin{equation*}
\|a^{\varepsilon}_{k_{n}}(-,\cdot, Du_{\varepsilon}) - a^{\varepsilon}(-,\cdot, Du_{\varepsilon})\|_{L^{\widetilde{B}}(Q)^{d}} \leq \dfrac{\varepsilon}{4 \|v\varphi\|_{L^{B}(0,T ; W^{1}_{0}L^{B}(\Omega))}} \;\; \textup{and}
\end{equation*}
\begin{equation*}
\left|\int_{Q}\left(a^{\varepsilon}(-,\cdot, Du_{\varepsilon k_{n}}) - a^{\varepsilon}(-,\cdot, Du_{\varepsilon})\right)\cdot Dv\varphi\, dxdt\right| < \dfrac{\varepsilon}{2}
\end{equation*}
for all integer $n\geq P$. Since 
\begin{equation*}
\begin{array}{l}
\left|\int_{Q}\left(a^{\varepsilon}_{k_{n}}(-,\cdot, Du_{\varepsilon k_{n}}) - a^{\varepsilon}(-,\cdot, Du_{\varepsilon})\right)\cdot Dv\varphi\, dxdt\right| \\  
\quad \leq \int_{-1}^{1} \left(\theta_{k_{n}}(\tau)\left|\int_{Q}\left(a^{\varepsilon}(-,t-\tau, Du_{\varepsilon k_{n}}) - a^{\varepsilon}(-,t-\tau, Du_{\varepsilon})\right)\cdot Dv\varphi\, dxdt\right| \right) d\tau  \\
\quad\quad + \left|\int_{Q}\left(a^{\varepsilon}_{k_{n}}(-,\cdot, Du_{\varepsilon}) - a^{\varepsilon}(-,\cdot, Du_{\varepsilon})\right)\cdot Dv\varphi\, dxdt\right|  \\
\quad \leq \int_{-1}^{1} \left|\int_{Q}\left(a^{\varepsilon}(-,t-\tau, Du_{\varepsilon k_{n}}) - a^{\varepsilon}(-,t-\tau, Du_{\varepsilon})\right)\cdot Dv\varphi\, dxdt\right|  d\tau  \\
\quad\quad + 2 \|a^{\varepsilon}_{k_{n}}(-,\cdot, Du_{\varepsilon}) - a^{\varepsilon}(-,\cdot, Du_{\varepsilon})\|_{L^{\widetilde{B}}(Q)^{d}} \|v\varphi\|_{L^{B}(0,T ; W^{1}_{0}L^{B}(\Omega))}, 
\end{array}
\end{equation*}
and $a^{\varepsilon}(-,t-\tau, Du_{\varepsilon})=a^{\varepsilon}(-,t-\tau, Du_{\varepsilon k_{n}})=0$ if $t-\tau \notin (0,T)$, we are led to 
\begin{equation*}
\left|\int_{Q}\left(a^{\varepsilon}_{k_{n}}(-,\cdot, Du_{\varepsilon k_{n}}) - a^{\varepsilon}(-,\cdot, Du_{\varepsilon})\right)\cdot Dv\varphi\, dxdt\right| \leq \varepsilon \;\; \textup{for \, all\, integer}\; n\geq P.
\end{equation*}
Therefore, as $n\to \infty$, we have 
\begin{equation}\label{c2eq16}
a^{\varepsilon}_{k_{n}}(-,\cdot, Du_{\varepsilon k_{n}}) \rightarrow a^{\varepsilon}(-,\cdot, Du_{\varepsilon})\;\, \textup{in} \; L^{\widetilde{B}}(Q)^{d}-\textup{weak}.
\end{equation}
Multiplying (\ref{c2eq13}) by $\varphi(t)$, $0< t< T$, $\varphi \in \mathcal{D}(0,T)$, and integrating from $0$ to $T$ (where $n$ is replaced by $k_{n}$), it follows
\begin{equation*}
-\int_{Q}u_{\varepsilon k_{n}}(t) v \varphi^{\prime}(t)\, dxdt + \int_{Q} a^{\varepsilon}_{k_{n}}\left(-,t, Du_{\varepsilon k_{n}}(t)\right)\cdot Dv\varphi(t)\, dxdt = \int_{Q}f_{k_{n}}(t), v\varphi(t)\, dxdt.
\end{equation*}
Passing to the limit, as $n\to \infty$, and using (\ref{c2eq15})-(\ref{c2eq16}) yields 
\begin{equation*}
-\int_{Q}u_{\varepsilon}(t) v \varphi^{\prime}(t)\, dxdt + \int_{Q} a^{\varepsilon}\left(-,t, Du_{\varepsilon}(t)\right)\cdot Dv\varphi(t)\, dxdt = \int_{Q}f(t)v\varphi(t)\, dxdt,
\end{equation*}
which drives us to 
\begin{equation}\label{c2eq17}
u^{\prime}_{\varepsilon}(t) + A^{\varepsilon}(t)u_{\varepsilon}(t) = f(t), \quad 0< t\leq T,
\end{equation}
where $A^{\varepsilon}(t)u = -\textup{div} a^{\varepsilon}(-,t, Du)$. It is obvious to prove that 
\begin{equation}\label{c2eq18}
u_{\varepsilon}(0)=0,
\end{equation}
and (\ref{c2eq17})-(\ref{c2eq18}) are equivalent to (\ref{c1eq2}). Assuming that $f=0$, it follows by (\ref{apenb3}) that $\|u^{\prime}_{\varepsilon}(t)\|_{L^{2}(\Omega)} \leq 0$ a.e. in $(0,T)$, that is, $u_{\varepsilon}=0$. Clearly, we have the following 
\begin{thm}\label{c3eq3}
	For each $\varepsilon>0$, there exists a unique function $u_{\varepsilon}$ from $[0,T]$ into $W^{1}_{0}L^{B}(\Omega;\mathbb{R})$ defined by (\ref{c1eq2}) and satisfying $u \in \mathcal{W}(0,T ; W^{1}_{0}L^{B}(\Omega;\mathbb{R})) \cap L^{\infty}(0,T ; L^{2}(\Omega;\mathbb{R}))$.
\end{thm}
\begin{rem}\label{c3rem1}
	We denote by $\mathcal{W}(0,T ; W^{1}_{0}L^{B}(\Omega;\mathbb{R}))$ the Banach space of functions $v \in L^{B}(0,T ; W^{1}_{0}L^{B}(\Omega;\mathbb{R}))$ such that $v^{\prime} \in  L^{\widetilde{B}}(0,T ; W^{-1}L^{\widetilde{B}}(\Omega;\mathbb{R}))$ for the norm $v \rightarrow \|v\|_{ L^{B}(0,T ; W^{1}_{0}L^{B}(\Omega;\mathbb{R}))} + \|v^{\prime}\|_{L^{\widetilde{B}}(0,T ; W^{-1}L^{\widetilde{B}}(\Omega;\mathbb{R}))}$. By (\ref{c1eq1}) the embedding $\mathcal{W}(0,T ; W^{1}_{0}L^{B}(\Omega;\mathbb{R})) \subset \mathcal{C}([0,T] ; L^{2}(\Omega;\mathbb{R}))$ is continuous.
\end{rem}
\par The reiterated periodic homogenization of problems (\ref{c1eq2}) amounts to find a homogenized problem, by using the reiterated two-scale convergence method (see \cite{tacha3,tacha4}), such that the sequence of solutions $u_{\varepsilon}$ converges to a limit $\textbf{u}$, which is precisely the solution of the homogenized problem. For this purpose, we need to recall the notion of reiterated two-scale convergence in $L^{B}(Q)$.

\section{Fundamentals of reiterated periodic homogenization}\label{sect4}

In all that follows, $B$ is the $N$-function in (\ref{c1eq1}), $Y=(0,1)^{d}$, $Z=(0,1)^{d}$, $\varTheta=(0,1)$ and $\Gamma=Y\times \varTheta$. The letter $\varepsilon$ throughout will denote a family of positive real numbers admitting $0$ as an accumulation point. When $\varepsilon=(\varepsilon_{n})_{n\in\mathbb{N}}$ with $0< \varepsilon_{n}\leq 1$ and $\varepsilon_{n}\rightarrow 0$ as $n\rightarrow\infty$, we will refer to $\varepsilon$ as a fundamental sequence. We put 
\begin{equation*}
\mathcal{C}_{per}(\Gamma\times Z) = \left\{v\in \mathcal{C}(\mathbb{R}^{d+1}_{y,\tau}\times\mathbb{R}^{d}_{z}) \, :\, v \,\textup{is}\, \Gamma\times Z-\textup{periodic} \right\},
\end{equation*}
\begin{equation*}
L^{B}_{per}(\Gamma\times Z) = \left\{v\in L^{B}_{loc}(\mathbb{R}^{d+1}_{y,\tau}\times\mathbb{R}^{d}_{z}) \, :\, v \,\textup{is}\, \Gamma\times Z-\textup{periodic} \right\},
\end{equation*}
then $L^{B}_{per}(\Gamma\times Z)$ is a Banach space under the Luxemburg norm $\|\cdot\|_{B,\Gamma\times Z}$, and $\mathcal{C}_{per}(\Gamma\times Z)$ is dense in $L^{B}_{per}(\Gamma\times Z)$. For $v\in L^{B}_{per}(\Gamma\times Z)$, letting 
\begin{equation*}
v^{\varepsilon}(x,t)= v\left(\dfrac{x}{\varepsilon},\dfrac{t}{\varepsilon},\dfrac{x}{\varepsilon^{2}}\right) \quad (x,t) \in \mathbb{R}^{d}\times\mathbb{R},
\end{equation*}
we have $v^{\varepsilon}\rightarrow \iint_{\Gamma\times Z} v(y,\tau,z) dyd\tau dz$ in $L^{B}(Q)$-weak as $\varepsilon\to 0$. Given
$v\in L^{B}_{loc}(Q\times \mathbb{R}^{d+1}_{y,\tau}\times\mathbb{R}^{d}_{z})$ and $\varepsilon>0$, we put  
\begin{equation}\label{c4eq1}
v^{\varepsilon}(x,t)= v\left(x,t,\dfrac{x}{\varepsilon},\dfrac{t}{\varepsilon},\dfrac{x}{\varepsilon^{2}}\right) \quad (x,t) \in Q,
\end{equation}
when $v\in \mathcal{C}(Q\times \mathbb{R}^{d+1}_{y,\tau}\times\mathbb{R}^{d}_{z})$. Since $\mathcal{C}(\overline{Q}; \mathcal{C}_{b}(\mathbb{R}^{d+1}_{y,\tau}\times\mathbb{R}^{d}_{z}))$ will legitimately be viewed as a subspace of $\mathcal{C}(Q\times \mathbb{R}^{d+1}_{y,\tau}\times\mathbb{R}^{d}_{z})$ (see \textbf{Apendix A}), as in \cite{tacha3} given $v\in \mathcal{C}(\overline{Q}; \mathcal{C}_{b}(\mathbb{R}^{d+1}_{y,\tau}\times\mathbb{R}^{d}_{z}))$, the trace $(x,t)\rightarrow v\left(x,t,\dfrac{x}{\varepsilon},\dfrac{t}{\varepsilon},\dfrac{x}{\varepsilon^{2}}\right)$ denoted by $v^{\varepsilon}$, is an element of $\mathcal{C}_{b}(Q)$. Moreover the transformation $v\rightarrow v^{\varepsilon}$ of $\mathcal{C}(\overline{Q}; \mathcal{C}_{b}(\mathbb{R}^{d+1}_{y,\tau}\times\mathbb{R}^{d}_{z}))$ into $\mathcal{C}_{b}(Q)$ extends by continuity to a unique linear mapping, still denoted $v\rightarrow v^{\varepsilon}$, of $L^{B}(Q; \mathcal{C}_{b}(\mathbb{R}^{d+1}_{y,\tau}\times\mathbb{R}^{d}_{z}))$  into $L^{B}(Q)$ with $\|v^{\varepsilon}\|_{B,Q} \leq \|v\|_{L^{B}(Q; \mathcal{C}_{b}(\mathbb{R}^{d+1}_{y,\tau}\times\mathbb{R}^{d}_{z}))}$ for all $v \in L^{B}(Q; \mathcal{C}_{b}(\mathbb{R}^{d+1}_{y,\tau}\times\mathbb{R}^{d}_{z}))$. 
\par The right-hand side of (\ref{c4eq1}) also makes sense when $v \in \mathcal{C}(\overline{Q}; L^{\infty}(\mathbb{R}^{d+1}_{y,\tau}\times\mathbb{R}^{d}_{z}))$. Moreover the mapping $v\rightarrow v^{\varepsilon}$ sends linearly and continuously $\mathcal{C}(\overline{Q}; L^{\infty}(\mathbb{R}^{d+1}_{y,\tau}\times\mathbb{R}^{d}_{z}))$ to $L^{\infty}(Q)$ with 
$\|v^{\varepsilon}\|_{L^{\infty}(Q)} \leq \sup_{(x,t)\in \overline{Q}} \|v(x,t)\|_{L^{\infty}(\mathbb{R}^{d})}$ for all $v\in \mathcal{C}(\overline{Q}; L^{\infty}(\mathbb{R}^{d+1}_{y,\tau}\times\mathbb{R}^{d}_{z}))$. \\
Recalling the spaces introduced in \textbf{Appendix A}, we start by defining reiterated two-scale convergence.
\begin{defn}
	A sequence of functions $(u_{\varepsilon})_{\varepsilon}\subset L^{B}(Q)$ is said to be :
	\begin{itemize}
		\item[(i)] weakly reiteratively two-scale convergent in $L^{B}(Q)$ to some $u_{0}\in L^{B}(Q; L^{B}_{per}(\Gamma\times Z))$ if, as $\varepsilon\to 0$, we have
		\begin{equation}\label{c4eq2}
		\int_{Q} u_{\varepsilon}\varphi^{\varepsilon} dxdt \rightarrow \iiint_{Q\times \Gamma\times Z} u_{0}\varphi dxdt dyd\tau dz
		\end{equation}
		for every $\varphi \in L^{\widetilde{B}}(Q; \mathcal{C}_{per}(\Gamma\times Z))$, where $\varphi^{\varepsilon}$ is defined as in (\ref{c4eq1}); 
		\item[(ii)] strongly reiteratively two-scale convergent in $L^{B}(Q)$ to some $u_{0}\in L^{B}(Q; L^{B}_{per}(\Gamma\times Z))$ if for $\eta>0$ and $\varphi \in L^{\widetilde{B}}(Q; \mathcal{C}_{per}(\Gamma\times Z))$ verifying $\|u_{0}-\varphi\|_{B,Q\times\Gamma\times Z} \leq \dfrac{\eta}{2}$ there exists $\rho>0$ such that 
		\begin{equation}\label{c4eq4}
		\|u_{\varepsilon}-\varphi^{\varepsilon}\|_{B,Q} \leq \eta\;\; \; \textup{for\, all}\; 0<\varepsilon\leq \rho.
		\end{equation}
	\end{itemize}
\end{defn}
When (\ref{c4eq2}) happens, we denote it by ``$u_{\varepsilon}\rightharpoonup u_{0}$ in $L^{B}(Q)$-weakly reiteratively 2s" and we say that $u_{0}$ is the weak reiterated two-scale limit in $L^{B}(Q)$ of the sequence $(u_{\varepsilon})_{\varepsilon}$. Also, when (\ref{c4eq4}) happens, we denote it by ``$u_{\varepsilon}\rightarrow u_{0}$ in $L^{B}(Q)$-strongly reiteratively 2s". 
\begin{rem}
	The above definition is given in the scalar setting, but it extends to vector valued functions, arguing in components.
\end{rem}
Since $B\in \Delta_{2}$, we have the following results whose proofs are an adaptation of those of \cite{tacha3}.
\begin{prop}
	If $u \in L^{B}(Q; \mathcal{C}_{per}(\Gamma\times Z))$, then (with the notation in (\ref{c4eq1})) $u^{\varepsilon}\rightharpoonup u$ in $L^{B}(Q)$-weakly reiteratively 2s, and we have 
	\begin{equation*}
	\lim_{\varepsilon\to 0} \|u^{\varepsilon}\|_{B,Q} = \|u\|_{B,Q\times\Gamma\times Z}.
	\end{equation*}
\end{prop}
Next, we give a sequential compactness result in the Orlicz setting.
\begin{prop}
	Given a bounded sequence $(u_{\varepsilon})_{\varepsilon} \subset L^{B}(Q)$, one can extract a not relabeled subsequence such that $(u_{\varepsilon})_{\varepsilon}$ is weakly reiteratively two-scale convergent in $L^{B}(Q)$.
\end{prop}
The results in the sequel follow as a consequence of density results in the `standard' setting.
\begin{prop}
	If a sequence $(u_{\varepsilon})_{\varepsilon}$ is weakly reiteratively two-scale convergent in $L^{B}(Q)$ to $u_{0} \in L^{B}(Q; L^{B}_{per}(\Gamma\times Z))$, then 
	\begin{itemize}
		\item[(i)] $u_{\varepsilon}\rightharpoonup \iint_{\varTheta\times Z} u_{0}(-,\cdot,\tau,z) d\tau dz$ in $L^{B}(Q)$-weakly two-scale\footnote{In the sense of \cite[Definition 4.1]{tacha1}}, and 
		 \item[(ii)] $u_{\varepsilon} \rightharpoonup \widetilde{u_{0}}$ in $L^{B}(Q)$-weakly as $\varepsilon\to 0$ where 
		 \begin{equation*}
		 \widetilde{u_{0}}(x,t) = \iint_{\Gamma\times Z} u_{0}(x,t, \cdot,\cdot,\cdot) dy d\tau dz.
		 \end{equation*}
	\end{itemize}
\end{prop}
In the sequel, we will consider the space 
\begin{equation}\label{c4eq5}
\mathfrak{X}^{B,\infty}_{per}(\mathbb{R}^{d+1}_{y,\tau}; \mathcal{C}_{b}(\mathbb{R}^{d}_{z})) := \mathfrak{X}^{B}_{per}(\mathbb{R}^{d+1}_{y,\tau}; \mathcal{C}_{b}(\mathbb{R}^{d}_{z})) \cap L^{\infty}(\mathbb{R}^{d+1}_{y,\tau}; \mathcal{C}_{b}(\mathbb{R}^{d}_{z})),
\end{equation}
endowed with the $L^{\infty}$ norm.
\begin{prop}\label{c4eq6}
	If $(u_{\varepsilon})_{\varepsilon}$ is weakly reiteratively two-scale convergent in $L^{B}(Q)$ to $u_{0} \in L^{B}(Q; L^{B}_{per}(\Gamma\times Z))$, then 
	\begin{equation*}
	\int_{Q} u_{\varepsilon} \varphi^{\varepsilon} dxdt \longrightarrow \iiint_{Q\times \Gamma\times Z} u_{0} \varphi dxdt dyd\tau dz,
	\end{equation*}
	for all $\varphi \in \mathcal{C}(\overline{Q})\otimes \mathfrak{X}^{B,\infty}_{per}(\mathbb{R}^{d+1}_{y,\tau}; \mathcal{C}_{b}(\mathbb{R}^{d}_{z}))$. 
	\par Moreover, if $v\in \mathcal{C}(\overline{Q};\mathfrak{X}^{B,\infty}_{per}(\mathbb{R}^{d+1}_{y,\tau}; \mathcal{C}_{b}(\mathbb{R}^{d}_{z})))$, then $v^{\varepsilon} \rightharpoonup v$ in $L^{B}(Q)$-weakly reiteratively 2s, as $\varepsilon\to 0$.
\end{prop}
\begin{rem}
	\begin{itemize}
		\item[(1)] If $v\in L^{B}(Q; \mathcal{C}_{per}(\Gamma\times Z))$, then $v^{\varepsilon}\rightarrow v$ in $L^{B}(Q)$-strongly reiteratively 2s, as $\varepsilon\to 0$.
		\item[(2)] If $(u_{\varepsilon})_{\varepsilon} \subset L^{B}(Q)$ is strongly reiteratively two-scale convergent in $L^{B}(Q)$ to $u_{0} \in L^{B}(Q; \mathcal{C}_{per}(\Gamma\times Z))$, then 
		\begin{itemize}
			\item[(i)] $u_{\varepsilon} \rightharpoonup u_{0}$ in $L^{B}(Q)$-weakly reiteratively 2s, as $\varepsilon\to 0$;
			\item[(ii)] $\|u_{\varepsilon}\|_{B,Q} \rightarrow \|u_{0}\|_{B,Q\times\Gamma\times Z}$ as $\varepsilon\to 0$.
		\end{itemize}
	\end{itemize}
\end{rem}
The following result is key to define weak reiterated two-scale convergence in Orlicz-Sobolev spaces, also providing a sequential compactness result in $W^{1}L^{B}(Q)$. It is an adaptation of \cite[Proposition 2.12]{tacha3}
\begin{prop}\label{c3eq4}
	Assume that $\varepsilon$ is a fundamental sequence, and that $(u_{\varepsilon})_{\varepsilon}$ is a bounded sequence in $L^{B}(0,T ; W^{1}L^{B}(\Omega))$. Then, a subsequence can be extracted, still denoted $\varepsilon$, such that as $\varepsilon\to 0$,
	\begin{equation}
	\begin{array}{l}
	u_{\varepsilon} \rightharpoonup u_{0} \quad \textup{in} \;\; L^{B}(0,T ; W^{1}L^{B}(\Omega))-\textup{weakly}, \\
	u_{\varepsilon} \rightharpoonup u_{0} \quad \textup{in} \;\; L^{B}(Q)-\textup{weakly\; reiteratively\; 2s}, \\
	\dfrac{\partial u_{\varepsilon}}{\partial x_{i}} \rightharpoonup \dfrac{\partial u_{0}}{\partial x_{i}} + \dfrac{\partial u_{1}}{\partial y_{i}} + \dfrac{\partial u_{2}}{\partial z_{i}} \quad \textup{in} \;\; L^{B}(Q)-\textup{weakly\; reiteratively\; 2s}\; (1\leq i\leq d),
	\end{array}
	\end{equation}
	where $u_{0} \in L^{B}(0,T ; W^{1}L^{B}(\Omega))$, $u_{1}\in L^{B}(Q\times\varTheta ; W^{1}_{\#}L^{B}_{per}(Y))$ and $u_{2} \in L^{B}(Q\times\varTheta ; L^{B}_{per}(Y; W^{1}_{\#}L^{B}(Z)))$. Futhermore if $(u_{\varepsilon})_{\varepsilon} \subset L^{B}(0,T ; W^{1}_{0}L^{B}(\Omega))$ then the weak reiterated two-scale limit $u_{0}$ lies in $L^{B}(0,T ; W^{1}_{0}L^{B}(\Omega))$. 
\end{prop}
In view of the next applications, we underline that, under the assumptions of the above proposition, the canonical injection $W^{1}L^{B}(Q) \hookrightarrow L^{B}(Q)$ is compact.
\begin{cor}
	If a sequence $(u_{\varepsilon})_{\varepsilon}$ is such that $u_{\varepsilon} \rightharpoonup v_{0}$ weakly reiteratively two-scale  in $W^{1}L^{B}(Q)$, we have  
	\begin{itemize}
		\item[(i)] $u_{\varepsilon}\rightharpoonup \iint_{\varTheta\times Z} v_{0}(-,\cdot,\tau,z) d\tau dz$ weakly two-scale in $W^{1}L^{B}(Q)$,
		\item[(ii)] $u_{\varepsilon} \rightharpoonup \widetilde{v_{0}}$ in $W^{1}L^{B}(Q)$-weakly as $\varepsilon\to 0$, where 
		\begin{equation*}
		\widetilde{v_{0}}(x,t) = \iint_{\Gamma\times Z} v_{0}(x,t, \cdot,\cdot,\cdot) dy d\tau dz.
		\end{equation*}
	\end{itemize}
\end{cor}
\begin{rem}
	If $(v_{\varepsilon})_{\varepsilon}\subset L^{B}(Q)$ and $v_{\varepsilon} \rightharpoonup v_{0}$ weakly reiteratively two-scale in $L^{B}(Q)$, as $\varepsilon\to 0$ and for any $\varepsilon>0 \, : \, v_{\varepsilon}\geq 0$ a.e. in $Q\times\Gamma\times Z$, then $v_{0} \geq 0$ a.e. in $Q\times\Gamma\times Z$.
\end{rem}
In the next section, we investigate to the reiterated homogenization of problem (\ref{c1eq2}).

\section{Homogenization results for (\ref{c1eq2})}\label{sect5}

We intend to investigate the limiting behavior, as $0<\varepsilon \to 0$, of a sequence of solutions $u_{\varepsilon}$ of (\ref{c1eq2}), taking into account all the assumptions in Section \ref{sect1}, in particular requiring that the coefficients $a$ satisfy (\ref{c1eq3})-(\ref{c1eq1}) together with the periodicity assumption (\ref{c1eq6}). Hence, we start by showing in suitable spaces of regular functions, where the compositions of functions in the weak formulation of (\ref{c1eq2}) are meaningful, whose convergence is appropriate to consider limits as $\varepsilon\to 0$. Let us begin by some preliminaries.

\subsection{Preliminaries}

Given $\textbf{v} = (v_{i}) \in \mathcal{C}_{per}(\Gamma\times Z; \mathbb{R})^{d}$, taking account traces results in Section \ref{sect4} and making use of the notation in (\ref{c4eq5}), under the extra assumption of periodicity on the first third variables of $a$ and for $\textbf{w} = (w_{i}) \in \mathcal{C}(\overline{Q};\mathcal{C}_{per}(\Gamma\times Z; \mathbb{R})^{d})$, we conclude that, for every $1\leq i\leq d$, the map $(x,t) \rightarrow a_{i}(-,\cdot,\textbf{w}(x,t,-,\cdot))$ belongs to $\mathcal{C}(\overline{Q}; \mathfrak{X}^{\widetilde{B},\infty}_{per}(\mathbb{R}^{d+1}_{y,\tau},\mathcal{C}_{b}(\mathbb{R}^{d}_{z})))$, where $a_{i}(-,\cdot,\textbf{v}(x,t,-,\cdot))$ is the function $(y,t,z)\rightarrow a_{i}(y,t,z, \textbf{v}(y,t,z))$. Therefore, for fixed $\varepsilon>0$, one defines the function $(x,t)\rightarrow a\left(\dfrac{x}{\varepsilon},\dfrac{t}{\varepsilon},\dfrac{x}{\varepsilon^{2}}, \textbf{v}\left(x,t,\dfrac{x}{\varepsilon},\dfrac{t}{\varepsilon},\dfrac{x}{\varepsilon^{2}}\right)\right)$ of $Q$ to $\mathbb{R}^{d}$ (denoted $a^{\varepsilon}(-,\cdot, \textbf{v}^{\varepsilon})$) as an element of $L^{\infty}(Q;\mathbb{R})^{d}$. 
Hence, we are in position to state our first convergence result.
\begin{prop}\label{c5eq3}
	For $\textbf{v} \in \mathcal{C}(\overline{Q};\mathcal{C}_{per}(\Gamma\times Z; \mathbb{R})^{d})$, let $a= (a_{i})_{1\leq i\leq d} : \mathbb{R}^{d}\times\mathbb{R}\times\mathbb{R}^{d}\times\mathbb{R}^{d} \rightarrow \mathbb{R}^{d}$ satisfy (\ref{c1eq3})-(\ref{c1eq6}), we have that 
	\begin{itemize}
		\item[(i)] for each $1\leq i \leq d$,
		\begin{equation*}
		\begin{array}{l}
		a_{i}(y,t,z, \textbf{v}(x,t,y,\tau,z)) \in \mathcal{C}(\overline{Q}; \mathfrak{X}^{\widetilde{B},\infty}_{per}(\mathbb{R}^{d+1}_{y,\tau},\mathcal{C}_{b}(\mathbb{R}^{d}_{z}))) \quad \textup{and} \\
		
		\\
		a_{i}^{\varepsilon}(-,\cdot, \textbf{v}^{\varepsilon}) \rightharpoonup a_{i}(-,\cdot, \textbf{v}) \quad \textup{in} \; L^{\widetilde{B}}(Q)-\textup{weakly\, reiteratively\, 2s,\; as}\; \varepsilon\to 0.
		\end{array}
		\end{equation*} 
		\item[(ii)] The function $\textbf{v}\rightarrow a(-,\cdot, \textbf{v})$ from $\mathcal{C}(\overline{Q}; \mathcal{C}_{per}(\Gamma\times Z; \mathbb{R})^{d})$ to $L^{\widetilde{B}}(Q; L^{\widetilde{B}}_{per}(\Gamma\times Z,\mathbb{R}))^{d}$ extends by continuity to a (not relabeled) function from $L^{B}(Q; L^{B}_{per}(\Gamma\times Z,\mathbb{R}))^{d}$ to $L^{\widetilde{B}}(Q; L^{\widetilde{B}}_{per}(\Gamma\times Z,\mathbb{R}))^{d}$, and satisfying
		\begin{equation}\label{c5eq1}
		\|a(-,\cdot, \textbf{v}) - a(-,\cdot, \textbf{w})\|_{L^{\widetilde{B}}(Q; L^{\widetilde{B}}_{per}(\Gamma\times Z,\mathbb{R}))^{d}} \leq c\, \|\textbf{v} - \textbf{w}\|^{\alpha}_{L^{B}(Q; L^{B}_{per}(\Gamma\times Z,\mathbb{R}))^{d}}
		\end{equation}
		and 
		\begin{equation}\label{c5eq2}
		\left(a(-,\cdot, \textbf{v}) - a(-,\cdot, \textbf{w})\, , \, \textbf{v}-\textbf{w}\right) \geq c_{2} B(|\textbf{v}-\textbf{w}|) \quad \textup{a.e.\, in}\; \overline{Q}\times\Gamma\times Z,
		\end{equation}
		for all $\textbf{v}, \textbf{w} \in L^{B}(Q; L^{B}_{per}(\Gamma\times Z,\mathbb{R}))^{d}$.
	\end{itemize}
\end{prop}
\begin{proof}
	It is easily seen that, for every $1\leq i \leq d$, when $\textbf{v} \in \mathcal{C}(\overline{Q};\mathcal{C}_{per}(\Gamma\times Z; \mathbb{R})^{d})$, $a_{i}(y,t,z, \textbf{v}(x,t,y,\tau,z)) \in \mathcal{C}(\overline{Q}; \mathfrak{X}^{\widetilde{B},\infty}_{per}(\mathbb{R}^{d+1}_{y,\tau},\mathcal{C}_{b}(\mathbb{R}^{d}_{z})))$. Thus, by Proposition \ref{c4eq6}, it follows that the sequence 
	\begin{equation*}
	a_{i}^{\varepsilon}(-,\cdot, \textbf{v}^{\varepsilon}) \rightharpoonup a_{i}(-,\cdot, \textbf{v}) \quad \textup{in} \; L^{\widetilde{B}}(Q)-\textup{weakly\, reiteratively\, 2s,\; as}\; \varepsilon\to 0.
	\end{equation*}
Moreover, by (\ref{c1eq3})-(\ref{c1eq6}), we get
\begin{equation*}
|a^{\varepsilon}(-,\cdot, \textbf{v}^{\varepsilon})- a^{\varepsilon}(-,\cdot, \textbf{v'}^{\varepsilon})|\leq c_{0}\,\widetilde{B}^{-1}(B(c_{1}|\textbf{v}^{\varepsilon}-\textbf{v'}^{\varepsilon}|)),
\end{equation*}
\begin{equation*}
\left(a^{\varepsilon}(-,\cdot, \textbf{v}^{\varepsilon}) - a^{\varepsilon}(-,\cdot, \textbf{v'}^{\varepsilon})\, , \, \textbf{v}^{\varepsilon}-\textbf{v'}^{\varepsilon}\right) \geq c_{2} B(|\textbf{v}^{\varepsilon}-\textbf{v'}^{\varepsilon}|) \quad \textup{a.e.\, in}\; \overline{Q},\; \textup{for\, any}\; \varepsilon>0.
\end{equation*}
On the other hand, taking into account (\ref{c1eq5}) and arguing as in the proof of Proposition \ref{c3eq2}, there exist constants $c_{0}, c_{2}$ such that (\ref{c5eq1}) and (\ref{c5eq2}) hold true for all $\textbf{v}, \textbf{w} \in \mathcal{C}(\overline{Q}; \mathcal{C}_{per}(\Gamma\times Z;\mathbb{R}))^{d}$. We end the proof by continuity and density arguments.
\end{proof}
\begin{cor}
	Under the same assumptions of Proposition \ref{c5eq3}. Let
	\begin{equation*}
	\phi_{\varepsilon}(x,t) = \varphi_{0}(x,t) + \varepsilon \varphi_{1}\left(x,\frac{x}{\varepsilon},\frac{t}{\varepsilon}\right) + \varepsilon^{2} \varphi_{2}\left(x,\frac{x}{\varepsilon},\frac{t}{\varepsilon},\frac{x}{\varepsilon^{2}}\right) \quad (\varepsilon>0,\;  (x,t) \in Q),
	\end{equation*}
	 where $\varphi_{0} \in \mathcal{D}(Q;\mathbb{R})$, $\varphi_{1} \in \mathcal{D}(Q;\mathbb{R})\otimes\mathcal{C}^{\infty}_{per}(\Gamma;\mathbb{R})$ and $\varphi_{2} \in \mathcal{D}(Q;\mathbb{R})\otimes\mathcal{C}^{\infty}_{per}(\Gamma;\mathbb{R})\otimes\mathcal{C}_{per}(Z)$ ($\phi_{\varepsilon} = \varphi_{0} + \varepsilon \varphi_{1}^{\varepsilon} + \varepsilon^{2} \varphi_{2}^{\varepsilon}$, shortly). Then, as $\varepsilon\to 0$, 
	 \begin{equation*}
	 a^{\varepsilon}(-,\cdot, D\phi_{\varepsilon}) \rightharpoonup a(-,\cdot, D\varphi_{0} + D_{y}\varphi_{1} + D_{z}\varphi_{2}) \quad \textup{in}\, L^{B}(Q)^{d}-\textup{weakly\, reiteratively\, 2s}.
	 \end{equation*}
	 Futhermore, given $(v_{\varepsilon})_{\varepsilon}\subset L^{B}(Q)^{d}$ such that $v_{\varepsilon}  \rightharpoonup v_{0}$ in $L^{B}(Q)^{d}$-weakly reiteratively 2s as $\varepsilon\to 0$, one has 
	 \begin{equation*}
	 \lim_{\varepsilon\to 0} \int_{Q} a^{\varepsilon}(-,\cdot, D\phi_{\varepsilon})\,v_{\varepsilon}\, dxdt = \iiint_{Q\times \Gamma\times Z} a(-,\cdot, D\varphi_{0} + D_{y}\varphi_{1} + D_{z}\varphi_{2})\,v_{0}\,dxdt dy d\tau dz.
	 \end{equation*} 
\end{cor}
\begin{proof}
	Let $\mathbf{\varphi}\in L^{B}(Q;\mathcal{C}_{per}(\Gamma\times Z;\mathbb{R}))^{d}$ and put 
	\begin{equation*}
	I(\varepsilon)= \int_{Q} a^{\varepsilon}(-,\cdot, D\phi_{\varepsilon})\,\mathbf{\varphi}^{\varepsilon}\, dxdt - \iiint_{Q\times \Gamma\times Z} a(-,\cdot, D\varphi_{0} + D_{y}\varphi_{1} + D_{z}\varphi_{2})\,\mathbf{\varphi}\,dxdt dy d\tau dz, \;\; \varepsilon>0.
	\end{equation*}
	we have $I(\varepsilon) = I_{1}(\varepsilon) + I2(\varepsilon)$, where 
	\begin{equation*}
	\begin{array}{l}
	I_{1}(\varepsilon) = \int_{Q} \left[a^{\varepsilon}(-,\cdot, D\phi_{\varepsilon}) - a^{\varepsilon}(-,\cdot, D\varphi_{0}+ (D_{y}\varphi_{1})^{\varepsilon} + (D_{z}\varphi_{2})^{\varepsilon})\right]\mathbf{\varphi}^{\varepsilon}\, dxdt \;\;\; \textup{and} \\
	I_{2}(\varepsilon) = \int_{Q} a^{\varepsilon}(-,\cdot, D\varphi_{0}+ (D_{y}\varphi_{1})^{\varepsilon} + (D_{z}\varphi_{2})^{\varepsilon})\mathbf{\varphi}^{\varepsilon}\, dxdt - \iiint_{Q\times \Gamma\times Z} a(-,\cdot, D\varphi_{0} + D_{y}\varphi_{1} + D_{z}\varphi_{2})\,\mathbf{\varphi}\,dxdt dy d\tau dz.
	\end{array}
	\end{equation*}
	By Proposition \ref{c4eq6}, we get $a^{\varepsilon}(-,\cdot, D\varphi_{0}+ (D_{y}\varphi_{1})^{\varepsilon} + (D_{z}\varphi_{2})^{\varepsilon}) \rightharpoonup a(-,\cdot, D\varphi_{0}+ D_{y}\varphi_{1} + D_{z}\varphi_{2})$ in  $L^{B}(Q)^{d}$-weakly reiteratively 2s when $\varepsilon\to 0$; that is, $\lim_{\varepsilon\to 0} I_{2}(\varepsilon)=0$. For $I_{1}(\varepsilon)$, since 
	\begin{equation*}
	D\phi_{\varepsilon} =  D\varphi_{0}+ (D_{y}\varphi_{1})^{\varepsilon} + (D_{z}\varphi_{2})^{\varepsilon} + \varepsilon(D_{x}\varphi_{1})^{\varepsilon} + \varepsilon(D_{y}\varphi_{2})^{\varepsilon} + \varepsilon^{2}(D_{x}\varphi_{2})^{\varepsilon},
	\end{equation*}
	 applying the H\"{o}lder's inequality (\ref{c2eq19}) and taking account Proposition \ref{c3eq2} yields 
	 \begin{equation*}
	 |I_{1}(\varepsilon)| \leq 2c \|\mathbf{\varphi}\|_{L^{B}(Q;\mathcal{C}_{per}(\Gamma\times Z;\mathbb{R}))^{d}} \|\varepsilon D_{x}\varphi_{1}^{\varepsilon} + \varepsilon D_{y}\varphi_{2}^{\varepsilon} + \varepsilon^{2} D_{x}\varphi_{2}^{\varepsilon}\|^{\rho}_{B,Q}
	 \end{equation*} 
	 and $ |I_{1}(\varepsilon)| \rightarrow 0$ as $\varepsilon\to 0$, and this proves the first statement. 
	 \par For the second limit, we repeat a similar decomposition where $\mathbf{\varphi}^{\varepsilon}$ is replaced by $v_{\varepsilon}$. Indeed one can define 
	 \begin{equation*}
	 I'(\varepsilon)= \int_{Q} a^{\varepsilon}(-,\cdot, D\phi_{\varepsilon})\,v_{\varepsilon}\, dxdt - \iiint_{Q\times \Gamma\times Z} a(-,\cdot, D\varphi_{0} + D_{y}\varphi_{1} + D_{z}\varphi_{2})\,v_{0}\,dxdt dy d\tau dz,
	 \end{equation*}
	 thus the proof will be concluded if we show that $\lim_{\varepsilon\to 0} I'(\varepsilon)=0$. However $I'(\varepsilon)= I^{\prime}_{1}(\varepsilon) + I^{\prime}_{2}(\varepsilon)$. Since $a(-,\cdot, D\varphi_{0} + D_{y}\varphi_{1} + D_{z}\varphi_{2})$ lies in $\mathcal{C}(\overline{Q}; \mathfrak{X}^{\widetilde{B},\infty}_{per}(\mathbb{R}^{d+1}_{y,\tau},\mathcal{C}_{b}(\mathbb{R}^{d}_{z})))$, by Proposition \ref{c4eq6} one gets $\lim_{\varepsilon\to 0} I^{'}_{2}(\varepsilon)=0$. By Proposition \ref{c3eq2}, we have 
	 \begin{equation*}
	 \|a^{\varepsilon}(-,\cdot, D\phi_{\varepsilon}) - a^{\varepsilon}(-,\cdot, D\varphi_{0}+ (D_{y}\varphi_{1})^{\varepsilon} + (D_{z}\varphi_{2})^{\varepsilon})\|_{\widetilde{B},Q} \leq c\,\|\varepsilon D_{x}\varphi_{1}^{\varepsilon} + \varepsilon D_{y}\varphi_{2}^{\varepsilon} + \varepsilon^{2} D_{x}\varphi_{2}^{\varepsilon}\|^{\rho}_{B,Q}.
	 \end{equation*}
	 Hence, $\lim_{\varepsilon\to 0} I^{\prime}_{1}(\varepsilon)=0$, and the proof is completed.
\end{proof}

\subsection{Main result}
Following the notation in Section \ref{sect2}, we set   
\begin{equation*}
\mathbb{F}^{1}_{0}L^{B} = \mathcal{W}_{0}\left(0,T; W^{1}_{0}L^{B}(\Omega;\mathbb{R})\right)\times L^{B}\left(Q\times\varTheta; W^{1}_{\#}L^{B}_{per}(Y;\mathbb{R})\right)\times L^{B}(Q\times\varTheta ; L^{B}_{per}(Y; W^{1}_{\#}L^{B}(Z)))
\end{equation*}
where $\mathcal{W}_{0}\left(0,T; W^{1}_{0}L^{B}(\Omega;\mathbb{R})\right) = \left\{v\in\mathcal{W}\left(0,T; W^{1}_{0}L^{B}(\Omega;\mathbb{R})\right) \, : \, v(0)=0 \right\}$ (see Remark \ref{c3rem1}), and then $\mathbb{F}^{1}_{0}L^{B}$ is a Banach space under the norm 
\begin{equation*}
\|\mathbf{u}\|_{\mathbb{F}^{1}_{0}L^{B}} = \|Du_{0}\|_{B,Q} + \|D_{y}u_{1}\|_{B,Q\times \Gamma} + \|D_{z}u_{2}\|_{B,Q\times \Gamma\times Z} \quad \left(\mathbf{u}=(u_{0},u_{1},u_{2})\in \mathbb{F}^{1}_{0}L^{B} \right).
\end{equation*}
Besides, thanks to the density of $\mathcal{D}(Q;\mathbb{R})$ in $\mathcal{W}\left(0,T; W^{1}_{0}L^{B}(\Omega;\mathbb{R})\right)$ and that of $\mathcal{C}^{\infty}_{per}(Y)/\mathbb{C}$ (resp. $\mathcal{C}^{\infty}_{per}(Z)/\mathbb{C}$) in $W^{1}_{\#}L^{B}_{per}(Y)$ (resp. $W^{1}_{\#}L^{B}_{per}(Z)$), the space 
\begin{equation*}
F^{\infty}_{0} = \mathcal{D}(Q;\mathbb{R}) \times \left[\mathcal{D}(Q;\mathbb{R})\otimes\mathcal{C}^{\infty}_{per}(\varTheta)\otimes\mathcal{C}^{\infty}_{per}(Y)/\mathbb{C} \right] \times \left[\mathcal{D}(Q;\mathbb{R})\otimes\mathcal{C}^{\infty}_{per}(\varTheta)\otimes\mathcal{C}^{\infty}_{per}(Y)/\mathbb{C}\otimes\mathcal{C}^{\infty}_{per}(Z)/\mathbb{C} \right]
\end{equation*}
is dense in $\mathbb{F}^{1}_{0}L^{B}$. Finally, for $\mathbf{v}=(v_{0},v_{1},v_{2})\in \mathbb{F}^{1}_{0}L^{B}$ denote $\mathbb{D}\mathbf{v} = Dv_{0} + D_{y}v_{1} + D_{z}v_{2}$, hypotheses (\ref{c1eq3})-(\ref{c1eq6}) drive to 
\begin{lem}\label{c5lem1}
	The variational problem
	\begin{equation}\label{c5lem2}
	\left\{\begin{array}{l}
	\mathbf{u}=(u_{0},u_{1},u_{2})\in \mathbb{F}^{1}_{0}L^{B} \\
	
	\\
	\int_{0}^{T} \left(u^{\prime}_{0}(t), v_{0}(t)\right) dt + \iiint_{Q\times \Gamma\times Z} a(-,\cdot, \mathbb{D}\mathbf{u})\cdot\mathbb{D}\mathbf{v}\, dxdt dyd\tau dz = \int_{0}^{T} \left(f(t), v_{0}(t)\right) dt,
	\end{array}\right.
	\end{equation}
	for all $\mathbf{v} = (v_{0}, v_{1}, v_{2}) \in \mathbb{F}^{1}_{0}L^{B}$, has at most one solution. 
\end{lem}
We are now in a position to state and prove the main result in the present section.
\begin{thm}
	Under hypoyheses (\ref{c1eq3})-(\ref{c1eq6}), and for each $\varepsilon>0$, let $u_{\varepsilon}$ be the unique solution to the initial boundary value problem (\ref{c1eq2}) (see Theorem \ref{c3eq3}). Then, as $\varepsilon\to 0$,
	\begin{equation}\label{c5eq5}
	u_{\varepsilon} \rightharpoonup u_{0} \;\;\; \textup{in}\; L^{B}(0,T; W^{1}_{0}L^{B}(\Omega;\mathbb{R}))-\textup{weakly},
	\end{equation}
	\begin{equation}\label{c5eq6}
	\dfrac{\partial u_{\varepsilon}}{\partial t} \rightharpoonup \dfrac{\partial u_{0}}{\partial t} \;\;\; \textup{in}\; L^{\widetilde{B}}(0,T; W^{-1}L^{\widetilde{B}}(\Omega;\mathbb{R}))-\textup{weakly},
	\end{equation}
	\begin{equation}\label{c5eq7}
	Du_{\varepsilon} \rightharpoonup \mathbb{D}\mathbf{u} = Du_{0} + D_{y}u_{1} + D_{z}u_{2} \;\;\; \textup{in}\; L^{B}(Q)^{d}-\textup{weakly\, reiteratively\, 2s}
	\end{equation}
	where $\mathbf{u}=(u_{0},u_{1},u_{2})\in \mathbb{F}^{1}_{0}L^{B}$ is the unique solution to the variational problem in Lemma \ref{c5lem1}.
\end{thm}
\begin{proof}
The first point is to check that the sequence $(u_{\varepsilon})_{\varepsilon>0}$ is bounded in $\mathcal{W}_{0}(0,T; W^{1}_{0}L^{B}(\Omega;\mathbb{R}))$. To this end, observe that, for $0<\varepsilon\leq 1$ arbitrarily fixed, and reffering to (\ref{c2eq17}), the initial problem (\ref{c1eq2}) implies the variational formulation 
\begin{equation}\label{c5eq4}
\left(u^{\prime}_{\varepsilon}(t), v\right) + \int_{\Omega} a^{\varepsilon}(-,\cdot, Du_{\varepsilon})\cdot Dv\, dx = \int_{\Omega} f(t)v\,dx
\end{equation}	
for all $v\in W^{1}_{0}L^{B}(\Omega;\mathbb{R})$. Taking $v=u_{\varepsilon}(t)$ $(0<t< T)$ and using properties (\ref{c1eq3})-(\ref{c1eq5}), it follows as in a priori estimate (\ref{apenb3}) (see Appendix B) : 
\begin{equation*}
\|u_{\varepsilon}(t)\|^{2}_{L^{2}(\Omega)} + c_{2}\int_{0}^{T} B\left(\|Du_{\varepsilon}(t)\|_{L^{B}(\Omega)}\right) dt \leq \int_{0}^{T} \widetilde{B}\left(\gamma\|f(t)\|_{W^{-1}L^{\widetilde{B}}(\Omega;\mathbb{R})}\right) dt,
\end{equation*}
where $\gamma=\frac{2}{c_{2}}$ if $0<c_{2}<1$ and $\gamma=2c_{2}$ if $c_{2}\geq 1$. Hence, by Lemma \ref{c2lem2}, one has 
\begin{equation*}
\|u_{\varepsilon}(t)\|^{2}_{L^{2}(\Omega)} + c_{2}\|u_{\varepsilon}(t)\|^{\rho}_{L^{B}(0,T; W^{1}_{0}L^{B}(\Omega;\mathbb{R}))} \leq \int_{0}^{T} \widetilde{B}\left(\gamma\|f(t)\|_{W^{-1}L^{\widetilde{B}}(\Omega;\mathbb{R})}\right) dt,
\end{equation*}
$(0<t<T)$ with $\rho=\rho_{1}$ or $\rho_{2}$. Thus 
\begin{equation*}
\sup_{0<\varepsilon\leq 1} \left(\|u_{\varepsilon}(t)\|_{L^{B}(0,T; W^{1}_{0}L^{B}(\Omega;\mathbb{R}))}\right) < \infty,
\end{equation*}
which, combined with (\ref{c1eq4}) and (\ref{c3eq1}), implies 
\begin{equation*}
\sup_{0<\varepsilon\leq 1} \left(\|a^{\varepsilon}(-,\cdot, Du_{\varepsilon})\|_{L^{\widetilde{B}}(Q)^{d}}\right) < \infty.
\end{equation*}
This being so, we deduce that
\begin{equation*}
\sup_{0<\varepsilon\leq 1}  \left(\left\|\dfrac{\partial u_{\varepsilon}}{\partial t} \right\|_{L^{\widetilde{B}}(0,T; W^{-1}L^{\widetilde{B}}(\Omega;\mathbb{R}))}  , \, \|\textup{div}a^{\varepsilon}(-,\cdot, Du_{\varepsilon})\|_{L^{\widetilde{B}}(0,T; W^{-1}L^{\widetilde{B}}(\Omega;\mathbb{R}))}\right) < \infty.
\end{equation*}
Therefore, the sequence $(u_{\varepsilon})_{0<\varepsilon \leq 1}$ is bounded in $\mathcal{W}_{0}(0,T; W^{1}_{0}L^{B}(\Omega;\mathbb{R}))$. Then, given an arbitrary fundamental sequence, there are a subsequence $\varepsilon$ and $\mathbf{u}=(u_{0},u_{1},u_{2}) \in \mathbb{F}_{0}^{1}L^{B}$ such that (\ref{c5eq5})-(\ref{c5eq7}) hold as $\varepsilon\to 0$ (see in particular Proposition \ref{c3eq4}).
\par Let us verify that $\mathbf{u}=(u_{0},u_{1},u_{2})$ is a solution of the varaiational problem of Lemma \ref{c5lem1}. Let $\phi = (\psi_{0},\psi_{1},\psi_{2}) \in F_{0}^{\infty}$ and $\phi_{\varepsilon}=\psi_{0}+\varepsilon\psi_{1}^{\varepsilon}+\varepsilon^{2}\psi_{2}^{\varepsilon}$. Then $\phi_{\varepsilon}\in \mathcal{D}(Q;\mathbb{R})$, so take $v=\phi_{\varepsilon}(t)$ ($0<t<T$) in (\ref{c5eq4}), integrate from 0 to $T$ and then use (\ref{c3cor1})
\begin{equation*}
\int_{0}^{T} \left(f(t)- u^{\prime}_{\varepsilon}(t),\, u_{\varepsilon}(t)-\phi_{\varepsilon}(t)\right) dt - \int_{Q} a^{\varepsilon}(-,\cdot, D\phi_{\varepsilon})\cdot (Du_{\varepsilon}-D\phi_{\varepsilon})\,dxdt \, \geq 0.
\end{equation*} 
Since $\int_{0}^{T} (u^{\prime}_{\varepsilon}(t), u_{\varepsilon}(t)) dt = \frac{1}{2}\|u_{\varepsilon}(T)\|^{2}_{L^{2}(\Omega)}$, the preceding inequality is equivalent to
\begin{equation}\label{c5eq8}
\begin{array}{l}
\int_{0}^{T} \left(f(t),\, u_{\varepsilon}(t)-\phi_{\varepsilon}(t)\right) dt + \int_{0}^{T} (u^{\prime}_{\varepsilon}(t), \phi_{\varepsilon}(t)) dt \\

\\
 - \int_{Q} a^{\varepsilon}(-,\cdot, D\phi_{\varepsilon})\cdot (Du_{\varepsilon}-D\phi_{\varepsilon})\,dxdt \, \geq \frac{1}{2}\|u_{\varepsilon}(T)\|^{2}_{L^{2}(\Omega)}.
\end{array}
\end{equation} 
Recalling that 
\begin{equation*}
\dfrac{\partial \phi_{\varepsilon}}{\partial x_{i}} \rightharpoonup \dfrac{\partial \psi_{0}}{\partial x_{i}} + \dfrac{\partial \psi_{1}}{\partial y_{i}} + \dfrac{\partial \psi_{2}}{\partial z_{i}} \;\; \textup{in}\, L^{B}(Q)-\textup{weakly\, reiteratively\, 2s} \;\; (1\leq i\leq d),
\end{equation*}
and
\begin{equation*}
\dfrac{\partial \phi_{\varepsilon}}{\partial t} \rightharpoonup \dfrac{\partial \psi_{0}}{\partial t} + \dfrac{\partial \psi_{1}}{\partial \tau}  \;\; \textup{in}\, L^{B}(Q)-\textup{weakly\, reiteratively\, 2s},
\end{equation*}
hence, $\phi_{\varepsilon} \rightarrow\psi_{0}$ in $L^{B}(0,T; W^{1}_{0}L^{B}(\Omega;\mathbb{R}))$-weakly, up to the subsequence extracted above (as $\varepsilon\to 0$) we have 
\begin{equation*}
\int_{0}^{T} \left(f(t),\, u_{\varepsilon}(t)-\phi_{\varepsilon}(t)\right) dt \rightarrow \int_{0}^{T} \left(f(t),\, u_{0}(t)-\psi_{0}(t)\right) dt,
\end{equation*}
\begin{equation*}
\begin{array}{rcl}
\int_{0}^{T} \left(u^{\prime}_{\varepsilon}(t),\,\phi_{\varepsilon}(t)\right) dt & = & - \int_{Q} u_{\varepsilon}\dfrac{\partial \phi_{\varepsilon}}{\partial t} dxdt \rightarrow - \int_{Q} u_{0}\dfrac{\partial \psi_{0}}{\partial t} dxdt \\
 & = & \int_{0}^{T} \left(u^{\prime}_{0}(t),\,\psi_{0}(t)\right) dt
\end{array}
\end{equation*}
\begin{equation*}
\int_{Q} a^{\varepsilon}(-,\cdot, D\phi_{\varepsilon})\cdot (Du_{\varepsilon}-D\phi_{\varepsilon})\,dxdt \rightarrow \iiint_{Q\times \Gamma\times Z} a(-\cdot, D\phi)\cdot \mathbb{D}(\mathbf{u}-\phi) dxdt dyd\tau dz.
\end{equation*}
Moreover, deducing by (\ref{c5eq5})-(\ref{c5eq6}) that $u_{\varepsilon} \rightarrow u_{0}$ in $\mathcal{W}_{0}(0,T; W^{1}_{0}L^{B}(\Omega;\mathbb{R}))$-weakly, and noting that the transformation $v\rightarrow \|v(T)\|_{L^{2}(\Omega)}$ from $\mathcal{W}_{0}(0,T; W^{1}_{0}L^{B}(\Omega;\mathbb{R}))$ into $\mathbb{R}$ is continuous, we are led to 
\begin{equation*}
\|u_{0}(T)\|^{2}_{L^{2}(\Omega)} \leq \liminf_{\varepsilon\to 0} \|u_{\varepsilon}(T)\|^{2}_{L^{2}(\Omega)}.
\end{equation*}
Therefore, passing to the $\liminf$ in (\ref{c5eq8}), when $\varepsilon\to 0$, and taking into account the arbitrariness of $\phi \in F^{\infty}_{0}$, and the density of $F^{\infty}_{0}$ in $F^{1}_{0}L^{B}$,  yields 
\begin{equation}\label{c5eq9}
\int_{0}^{T} \left(f(t)-u^{\prime}_{0}(t),\, u_{0}(t)-v_{0}(t)\right) dt - \iiint_{Q\times \Gamma\times Z} a(-,\cdot, \mathbb{D}\mathbf{v})\cdot \mathbb{D}(\textbf{u}-\textbf{v})\,dxdtdyd\tau dz \geq 0
\end{equation}
for all $\textbf{v} = (v_{0},v_{1},v_{2}) \in F^{1}_{0}L^{B}$. \\
Now, in (\ref{c5eq9}), choosing $\textbf{v}:= \textbf{u}- \zeta \textbf{w}$, with $\textbf{w} = (w_{0},w_{1},w_{2}) \in F^{1}_{0}L^{B}$, and dividing by $\zeta>0$, we get : 
\begin{equation*}
\int_{0}^{T} \left(f(t)-w^{\prime}_{0}(t),\, w_{0}(t)\right) dt - \iiint_{Q\times \Gamma\times Z} a(-,\cdot, \mathbb{D}\mathbf{u}-\zeta\mathbb{D}\mathbf{w})\cdot \mathbb{D}\textbf{v}\,dxdtdyd\tau dz \geq 0.
\end{equation*}
Using the continuity of $a$ in its last argument and passing to the limit as $\zeta\to 0$, we are led to 
\begin{equation*}
\int_{0}^{T} \left(f(t)-w^{\prime}_{0}(t),\, w_{0}(t)\right) dt - \iiint_{Q\times \Gamma\times Z} a(-,\cdot, \mathbb{D}\mathbf{u})\cdot \mathbb{D}\textbf{v}\,dxdtdyd\tau dz \geq 0,
\end{equation*}
that is for all $\textbf{v} \in F^{1}_{0}L^{B}$. Thus, if we replace $\textbf{v} = (v_{0},v_{1},v_{2}) \in F^{1}_{0}L^{B}$ by $\textbf{v}^{1} = (-v_{0},-v_{1},-v_{2})$, we deduce 
\begin{equation*}
-\int_{0}^{T} \left(f(t)-w^{\prime}_{0}(t),\, w_{0}(t)\right) dt \geq  \iiint_{Q\times \Gamma\times Z} a(-,\cdot, \mathbb{D}\mathbf{u})\cdot \mathbb{D}\textbf{v}^{1}\,dxdtdyd\tau dz, \;\;\; \textup{i.e.},
\end{equation*}
\begin{equation}\label{c5eq10}
-\int_{0}^{T} \left(f(t)-w^{\prime}_{0}(t),\, w_{0}(t)\right) dt \geq  -\iiint_{Q\times \Gamma\times Z} a(-,\cdot, \mathbb{D}\mathbf{u})\cdot \mathbb{D}\textbf{v}\,dxdtdyd\tau dz, \;\;\; \textup{i.e.},
\end{equation}
thus by (\ref{c5eq9}) and (\ref{c5eq10}), the equality follows, i.e. $\textbf{u} = (u_{0},u_{1},u_{2})$ verifies (\ref{c5lem1}). 
\end{proof}
\begin{rem}\label{rem1}
	 Under hypotheses (\ref{c1eq5})-(ii), one can prove that the solution $\mathbf{u}=(u_{0},u_{1},u_{2})\in \mathbb{F}^{1}_{0}L^{B}$ of (\ref{c5lem2}) is unique.  Let $\mathbf{w}=(w_{0},w_{1},w_{2})\in \mathbb{F}^{1}_{0}L^{B}$ be another solution of (\ref{c5lem2}). On the one hand, we have : 
	\begin{equation*}
	\begin{array}{l}
	-\int_{0}^{T} \left(u^{\prime}_{0}(t), w_{0}(t)\right) dt - \iiint_{Q\times \Gamma\times Z} a(-,\cdot, \mathbb{D}\mathbf{u})\cdot\mathbb{D}\mathbf{w}\, dxdt dyd\tau dz  = -\int_{0}^{T} \left(f(t), w_{0}(t)\right) dt, \\
	
	\\
	\int_{0}^{T} \left(u^{\prime}_{0}(t), u_{0}(t)\right) dt + \iiint_{Q\times \Gamma\times Z} a(-,\cdot, \mathbb{D}\mathbf{u})\cdot\mathbb{D}\mathbf{u}\, dxdt dyd\tau dz  = \int_{0}^{T} \left(f(t), u_{0}(t)\right) dt, \\
	
	\\
	\int_{0}^{T} \left(w^{\prime}_{0}(t), w_{0}(t)\right) dt + \iiint_{Q\times \Gamma\times Z} a(-,\cdot, \mathbb{D}\mathbf{w})\cdot\mathbb{D}\mathbf{w}\, dxdt dyd\tau dz  = \int_{0}^{T} \left(f(t), w_{0}(t)\right) dt, \\
	
	\\
	-\int_{0}^{T} \left(w^{\prime}_{0}(t), u_{0}(t)\right) dt - \iiint_{Q\times \Gamma\times Z} a(-,\cdot, \mathbb{D}\mathbf{w})\cdot\mathbb{D}\mathbf{u}\, dxdt dyd\tau dz  = -\int_{0}^{T} \left(f(t), u_{0}(t)\right) dt.
	\end{array}
	\end{equation*}	
	Thus, we have that 
	\begin{equation*}
	\int_{0}^{T} \left(u^{\prime}_{0}(t)-w^{\prime}_{0}(t), u_{0}(t)-w_{0}(t)\right) dt  + \iiint_{Q\times \Gamma\times Z} \left(a(-,\cdot, \mathbb{D}\mathbf{u})-a(-,\cdot, \mathbb{D}\mathbf{w})\right)\cdot(\mathbb{D}\mathbf{u}- \mathbb{D}\mathbf{w})\, dxdt dyd\tau dz = 0.
	\end{equation*}
	Since $\int_{0}^{T} \left(u^{\prime}_{0}(t)-w^{\prime}_{0}(t), u_{0}(t)-w_{0}(t)\right) dt = \frac{1}{2}\|u_{0}(T)-w_{0}(T)\|^{2}_{L^{2}(\Omega)}$ and by (\ref{c1eq5})-(ii) we get
	\begin{equation*}
	\begin{array}{rcl}
	0 & = & \frac{1}{2}\|u_{0}(T)-w_{0}(T)\|^{2}_{L^{2}(\Omega)}  + \iiint_{Q\times \Gamma\times Z} \left(a(-,\cdot, \mathbb{D}\mathbf{u})-a(-,\cdot, \mathbb{D}\mathbf{w})\right)\cdot(\mathbb{D}\mathbf{u}- \mathbb{D}\mathbf{w})\, dxdt dyd\tau dz  \\
	 & = & \iiint_{Q\times \Gamma\times Z} \left(a(-,\cdot, \mathbb{D}\mathbf{u})-a(-,\cdot, \mathbb{D}\mathbf{w})\right)\cdot(\mathbb{D}\mathbf{u}- \mathbb{D}\mathbf{w})\, dxdt dyd\tau dz \\
	  & \geq & c_{2} \iiint_{Q\times \Gamma\times Z} B\left(|\mathbb{D}\mathbf{u}- \mathbb{D}\mathbf{w}|\right)\, dxdt dyd\tau dz.
	\end{array}
	\end{equation*}
	Hence, by the standard arguments (see \cite[Remark 4.4]{tacha4}), we conclude that $\textbf{w}=\textbf{u}$.
\end{rem}
The variational problem in Lemma \ref{c5lem1} is referred to as the \textit{global homogenized problem} for (\ref{c1eq2}). The term \textit{global} is used here to lay emphasis on the fact (\ref{c5lem2}) includes both the local (or microscopic) equations for $u_{1}$ and $u_{2}$, and the macroscopic homogenized equation for $u_{0}$. In the next section, we shall study the \textit{macroscopic homogenized problem}.

\subsection{Macroscopic Homogenized Problem}

Under the periodic hypotheses (\ref{c1eq6}), the global homogenized problem (\ref{c5lem2}) is equivalent to the following three systems\footnote{see \cite{tacha4} for the elliptic case} : 
\begin{equation}\label{c5eq11}
\begin{array}{c}
 \iiint_{Q\times \Gamma\times Z} a(-,\cdot, Du_{0}+D_{y}u_{1}+D_{z}u_{2})\cdot D_{z}v_{2}\, dxdt dyd\tau dz = 0 \\
 
 \\
 \textup{for \, all}\; v_{2} \in L^{B}\left(Q\times\varTheta; L^{B}_{per}(Y;W^{1}_{\#}L^{B}(Z))\right),
\end{array}
\end{equation}
\begin{equation}\label{c5eq12}
\begin{array}{c}
\iint_{Q\times \Gamma} \left(\int_{Z} a(-,\cdot, Du_{0}+D_{y}u_{1}+D_{z}u_{2})\,dz\right)\cdot D_{y}v_{1}\, dxdt dyd\tau  = 0 \\

\\
\textup{for \, all}\; v_{1} \in L^{B}_{per}\left(Q\times\varTheta; W^{1}_{\#}L^{B}(Y)\right),
\end{array}
\end{equation}
\begin{equation}\label{c5eq13}
\begin{array}{c}
	\int_{0}^{T} \left(u^{\prime}_{0}(t), v_{0}(t)\right) dt + \int_{Q} \left(\iint_{\Gamma\times Z} a(-,\cdot, Du_{0}+D_{y}u_{1}+D_{z}u_{2}) dyd\tau dz \right)\cdot Dv_{0}(t)\, dxdt \\
	
	\\
	 = \int_{0}^{T} \left(f(t), v_{0}(t)\right) dt \\

\\
\textup{for \, all}\; v_{0} \in \mathcal{W}_{0}\left(0,T; W^{1}_{0}L^{B}(\Omega;\mathbb{R})\right).
\end{array}
\end{equation}
Now, we are in position to derive a macroscopic homogenized problem (se \cite[Page 23]{kenne1} for the elliptic case). Hence, let  $\xi\in\mathbb{R}^{d}$ be arbitrarily fixed. For a.e. $y\in Y$, consider the following variational cell problem whose solution is denoted by $\pi_{2}(y,\xi)$ : 
\begin{equation}\label{c5eq14}
\left\{\begin{array}{l}
 \textup{find}\; \pi_{2}(y,\xi) \in L^{B}(\varTheta; W^{1}_{\#}L^{B}(Z)) \; \textup{such\, that} \\
 
 \\
 \iint_{\varTheta\times Z} a(y,z,\tau, \xi + D_{z}\pi_{2}(y,\xi))\cdot D_{z}\theta\, dzd\tau = 0 \;\;\, \textup{for\, all}\, \theta\in L^{B}(\varTheta; W^{1}_{\#}L^{B}(Z)). 
\end{array}\right.
\end{equation}
This problem has a unique solution (see \cite{kenne1,moha1,yama1}). \\
Comparing (\ref{c5eq13}) with \ref{c5eq14} for $\xi = Du_{0}(x,t)+D_{y}u_{1}(x,t,y)$, we can consider 
\begin{equation*}
Q\times \mathbb{R}^{d}_{y}\ni (x,t,y) \rightarrow \pi_{2}(y, Du_{0}(x,t)+D_{y}u_{1}(x,t,y)) \in W^{1}_{\#}L^{B}(Z).
\end{equation*}
Hence, defining $h$, for a.e. $y\in Y$, and for any $\xi\in\mathbb{R}^{d}$ by 
\begin{equation*}
h(y,\tau,\xi) := \iint_{\varTheta\times Z} a_{i}(y,z,\tau, \xi + D_{z}\pi_{2}(y,\xi))\,dzd\tau,
\end{equation*} 
(\ref{c5eq12}) becomes 
\begin{equation*}
\iint_{Q\times \Gamma} h(y, \tau, Du_{0}(x,t)+D_{y}u_{1}(x,t,y))\cdot D_{y}v_{1} \, dxdt dyd\tau = 0,
\end{equation*}
for all $v_{1} \in L^{B}_{per}\left(Q\times\varTheta; W^{1}_{\#}L^{B}(Y)\right)$. \\
Consequently, in analogy with the previous steps, one can consider, for any $\xi\in\mathbb{R}^{d}$, the function $\pi_{1}\in L^{B}(\varTheta; W^{1}_{\#}L^{B}(Y))$ solution of the following variational problem (unique since \ref{c1eq5} (ii) holds) : 
\begin{equation}\label{c5eq15}
\left\{\begin{array}{l}
\textup{find}\; \pi_{1}(\xi) \in L^{B}(\varTheta; W^{1}_{\#}L^{B}(Y)) \; \textup{such\, that} \\

\\
\int_{\Gamma} h(y,\tau, \xi + D_{y}\pi_{1}(\xi))\cdot D_{y}\theta\, dyd\tau = 0 \;\;\, \textup{for\, all}\, \theta\in L^{B}(\varTheta; W^{1}_{\#}L^{B}(Y)). 
\end{array}\right.
\end{equation}
Note also that (\ref{c5eq12}) lead us to $u_{1}=\pi_{1}(Du_{0})$. Set, again, for $\xi\in\mathbb{R}^{d}$ 
\begin{equation}\label{c5eq16}
q(\xi) = \int_{\Gamma} h(y,\tau, \xi + D_{y}\pi_{1}(\xi))\,dyd\tau,
\end{equation}
the function $q$ above is well defined. Moreover, it results in (\ref{c5eq13}) and the above cell problems that 
\begin{equation}\label{c5eq17}
\begin{array}{l}
\int_{0}^{T} \left(u^{\prime}_{0}(t), v_{0}(t)\right) dt + \int_{Q} q(Du_{0})\cdot Dv_{0}(t)\, dxdt  = \int_{0}^{T} \left(f(t), v_{0}(t)\right) dt \\

\\
\textup{for \, all}\; v_{0} \in \mathcal{W}_{0}\left(0,T; W^{1}_{0}L^{B}(\Omega;\mathbb{R})\right).
\end{array}
\end{equation}
Finally, as in the elliptic case, we have the following theorem which give us the macroscopic homogenized problem.
\begin{thm}
	For every $\varepsilon>0$, let (\ref{c1eq2}) be such that $a$ and $f$ satisfy (\ref{c1eq1})-(\ref{c1eq6}). Let $u_{0} \in \mathcal{W}_{0}\left(0,T; W^{1}_{0}L^{B}(\Omega;\mathbb{R})\right)$ be the solution defined by means of (\ref{c5lem2}). Then, there is a unique solution of the macroscopic homogenized problem 
	\begin{equation}\label{c5eq18}
	\left\{\begin{array}{l}
	\dfrac{\partial u_{0}}{\partial t} - \textup{div}\,q(Du_{0}) = f \quad \textup{in}\;Q \\
	
	\\
	u_{0}=0 \quad \textup{on}\; \partial\Omega\times (0,T) \\
	
	\\
	u_{0}(x,0) = 0 \quad \textup{in} \; \Omega,
	\end{array}\right.
	\end{equation}
	where $q$ is defined by (\ref{c5eq16}).
\end{thm}
\begin{proof}
From (\ref{c5eq17}), the function $u_{0}$ is a solution of (\ref{c5eq18}). Let $w_{0} \in \mathcal{W}_{0}\left(0,T; W^{1}_{0}L^{B}(\Omega;\mathbb{R})\right)$ be another solution of (\ref{c5eq18}), then we have 
\begin{equation*}
\begin{array}{l}
-\int_{0}^{T} \left(u^{\prime}_{0}(t), w_{0}(t)\right) dt - \int_{Q} q(Du_{0})\cdot Dw_{0}(t)\, dxdt  = -\int_{0}^{T} \left(f(t), w_{0}(t)\right) dt, \\

\\
\int_{0}^{T} \left(u^{\prime}_{0}(t), u_{0}(t)\right) dt + \int_{Q} q(Du_{0})\cdot Du_{0}(t)\, dxdt  = \int_{0}^{T} \left(f(t), u_{0}(t)\right) dt, \\

\\
\int_{0}^{T} \left(w^{\prime}_{0}(t), w_{0}(t)\right) dt + \int_{Q} q(Dw_{0})\cdot Dw_{0}(t)\, dxdt  = \int_{0}^{T} \left(f(t), w_{0}(t)\right) dt, \\

\\
-\int_{0}^{T} \left(w^{\prime}_{0}(t), u_{0}(t)\right) dt - \int_{Q} q(Dw_{0})\cdot Du_{0}(t)\, dxdt  = -\int_{0}^{T} \left(f(t), u_{0}(t)\right) dt.
\end{array}
\end{equation*}	
Thus, we have that 
\begin{equation*}
 \int_{0}^{T} \left(u^{\prime}_{0}(t)-w^{\prime}_{0}(t), u_{0}(t)-w_{0}(t)\right) dt  + \int_{Q} \left(q(Du_{0}) -  q(Dw_{0})\right)\cdot (Du_{0} - Dw_{0})\, dxdt = 0.
\end{equation*}
Replacing $q$ by (\ref{c5eq16}), we get 
\begin{equation*}
\begin{array}{l}
\int_{0}^{T} \left(u^{\prime}_{0}(t)-w^{\prime}_{0}(t), u_{0}(t)-w_{0}(t)\right) dt  \\

\\
\quad + \int_{Q}\int_{\Gamma} \left( h(y,\tau, Du_{0} + D_{y}\pi_{1}(Du_{0})) - h(y,\tau, Dw_{0} + D_{y}\pi_{1}(Dw_{0})) \right)\cdot (Du_{0} - Dw_{0})\, dxdt \,dyd\tau = 0,
\end{array}
\end{equation*}
i.e., 
\begin{equation*}
\begin{array}{l}
\int_{0}^{T} \left(u^{\prime}_{0}(t)-w^{\prime}_{0}(t), u_{0}(t)-w_{0}(t)\right) dt  \\

\\
\quad + \iiint_{Q\times \Gamma\times Z} \bigg[a\left(y,z,\tau,u_{0},\substack{\underbrace{Du_{0}+ D_{y}\pi_{1}(Du_{0}) + D_{z}\pi_{2}(y,\tau, Du_{0}+D_{y}\pi_{1}(Du_{0}))} \\ =A} \right) \\

\\
\qquad \quad - a\left(y,z,\tau,w_{0},\substack{\underbrace{Dw_{0}+ D_{y}\pi_{1}(Dw_{0}) + D_{z}\pi_{2}(y,\tau, Dw_{0}+D_{y}\pi_{1}(Dw_{0}))} \\ =B} \right) \bigg] \\

\\
\qquad \quad \quad \cdot \left[(Du_{0}-Dw_{0}) \right] \, dxdt dyd\tau dz = 0.
\end{array}
\end{equation*}
Thus, 
\begin{equation*}
\begin{array}{l}
\int_{0}^{T} \left(u^{\prime}_{0}(t)-w^{\prime}_{0}(t), u_{0}(t)-w_{0}(t)\right) dt  \\

\\
\quad + \iiint_{Q\times \Gamma\times Z} \left(a(y,z,\tau, u_{0}, A) - a(y,z,\tau, w_{0}, B) \right)\cdot (A-B)\, dxdtdyd\tau dz = 0.
\end{array}
\end{equation*}
Since 
\begin{equation*}
\begin{array}{rcl}
0 & = &\int_{0}^{T} \left(u^{\prime}_{0}(t)-w^{\prime}_{0}(t), u_{0}(t)-w_{0}(t)\right) dt  \\

\\
&  &  \quad + \iiint_{Q\times \Gamma\times Z} a(y,z,\tau, u_{0}, A)\cdot (A-Du_{0})\, dxdtdyd\tau dz \\

\\
 & = & \int_{0}^{T} \left(u^{\prime}_{0}(t)-w^{\prime}_{0}(t), u_{0}(t)-w_{0}(t)\right) dt  \\
 
 \\
&  &  \quad + \iiint_{Q\times \Gamma\times Z} a(y,z,\tau, u_{0}, B)\cdot (B-Dw_{0})\, dxdtdyd\tau dz  
\end{array}
\end{equation*}
and 
\begin{equation*}
\begin{array}{rcl}
0 & = &\int_{0}^{T} \left(u^{\prime}_{0}(t)-w^{\prime}_{0}(t), u_{0}(t)-w_{0}(t)\right) dt  \\

\\
  &  & \quad + \iiint_{Q\times \Gamma\times Z} a(y,z,\tau, u_{0}, A)\cdot (B-Dw_{0})\, dxdtdyd\tau dz \\

\\
& = & \int_{0}^{T} \left(u^{\prime}_{0}(t)-w^{\prime}_{0}(t), u_{0}(t)-w_{0}(t)\right) dt  \\

\\
&  &  \quad + \iiint_{Q\times \Gamma\times Z} a(y,z,\tau, u_{0}, B)\cdot (A-Du_{0})\, dxdtdyd\tau dz.  
\end{array}
\end{equation*}
Thus, using similar arguments as in Remark \ref{rem1}, relying on (\ref{c1eq5})-(ii) give uniqueness.
\end{proof}
\begin{rem}
	Following this work, if we consider $B(t) = \frac{t^{p}}{p}$ ($p>1, \;t\geq 0$), then $B \in \Delta_{2}\cap\Delta'$ (with $\widetilde{B}(t)=\frac{t^{q}}{q}$, $q=\frac{p}{p-1}$) and one has $W^{1}L^{B}(\Omega,\mathbb{R}) \equiv W^{1,p}(\Omega,\mathbb{R})$, $L^{B}(0,T ; W^{1}L^{B}(\Omega,\mathbb{R})) \equiv L^{p}(0,T ; W^{1,p}(\Omega,\mathbb{R}))$ and $L^{\widetilde{B}}(0,T ; W^{-1}L^{\widetilde{B}}(\Omega,\mathbb{R})) \equiv L^{q}(0,T ; W^{-1,q}(\Omega,\mathbb{R}))$.  Therefore, reiterated homogenization problem (\ref{c1eq2}) can be rewritten in the classical Sobolev's spaces (see e.g. \cite{wou2}).
\end{rem}
 
\section{Appendix}
\subsection{Appendix A : Traces results}

This subsection is devoted to recall some results which are crucial for reiterated multiple scales convergence in the Orlicz setting. The notation is similar to \cite[section 2]{tacha3}.  
\par Traces of the form 
\begin{equation*}
v^{\varepsilon}(x) := v\left(x,\frac{x}{\varepsilon}, \frac{x}{\varepsilon^{2}}\right), \quad x\in U, \;\; \varepsilon >0,
\end{equation*}
when $v \in \mathcal{C}(U\times \mathbb{R}^{d}_{y}\times \mathbb{R}^{d}_{z})$ are well known and, clearly the operator of order $\varepsilon >0$, $(t^{\varepsilon})$, defined by 
\begin{equation*}
t^{\varepsilon} : v \in \mathcal{C}(U\times \mathbb{R}^{d}_{y}\times \mathbb{R}^{d}_{z}) \rightarrow v^{\varepsilon} \in \mathcal{C}(U), 
\end{equation*}
is linear and continuous. \\
Making use of the subscript $_{b}$ to denote bounded functions, the same definitions and properties hold true (since $\overline{U}$ is compact), when 
\begin{equation*}
v \in \mathcal{C}(\overline{U}; \mathcal{C}_{b}(\mathbb{R}^{d}_{y}\times \mathbb{R}^{d}_{z})) \subset \mathcal{C}(\overline{U}; \mathcal{C}(\mathbb{R}^{d}_{y}\times \mathbb{R}^{d}_{z})) \cong \mathcal{C}(\overline{U}\times\mathbb{R}^{d}_{y}\times \mathbb{R}^{d}_{z}).
\end{equation*}
Then, considering $\mathcal{C}(\overline{U}; \mathcal{C}_{b}(\mathbb{R}^{d}_{y}\times \mathbb{R}^{d}_{z}))$ as a subspace of $\mathcal{C}(\overline{U}\times\mathbb{R}^{d}_{y}\times \mathbb{R}^{d}_{z})$, $v^{\varepsilon}\in \mathcal{C}_{b}(U)$ and, with an abuse of notation the operator $t^{\varepsilon}$ can be interpreted from $\mathcal{C}(\overline{U}; \mathcal{C}_{b}(\mathbb{R}^{d}_{y}\times \mathbb{R}^{d}_{z}))$ to $\mathcal{C}_{b}(U)$ as linear and continuous. Moreover, it is easily seen that 
\begin{equation}\label{c2eq1}
\left|v^{\varepsilon}(x) \right| = \left|v(x, \frac{x}{\varepsilon}, \frac{x}{\varepsilon^{2}}) \right| \leq \|v(x)\|_{\infty},
\end{equation}
for every $x\in U$. By $v \in L^{\Phi}(U; \mathcal{C}_{b}(\mathbb{R}^{d}_{y}\times \mathbb{R}^{d}_{z}))$, we mean that the function $x\rightarrow \|v(x)\|_{\infty}$, from $U$ into $\mathbb{R}$, belongs to $L^{\Phi}(U)$ and 
\begin{equation*}
\|v\|_{L^{\Phi}(U; \mathcal{C}_{b}(\mathbb{R}^{d}_{y}\times \mathbb{R}^{d}_{z}))} = \inf\left\{k>0\, :\; \int_{U} \Phi\left(\dfrac{\|v(x)\|_{\infty}}{k} \right) dx \leq 1 \right\} < +\infty.
\end{equation*}
Recalling that $N$-functions are non decreasing, from (\ref{c2eq1}), we deduce that :
\begin{equation*}
\Phi\left(\dfrac{|v^{\varepsilon}(x)|}{k}\right) \leq \Phi\left(\dfrac{\|v^{\varepsilon}(x)\|_{\infty}}{k} \right), \;\; \textup{for\, all}\, k>0, \; \textup{for all}\, x\in \overline{U},
\end{equation*}
\begin{equation*}
\int_{U}\Phi\left(\dfrac{|v^{\varepsilon}(x)|}{k}\right)dx \leq \int_{U}\Phi\left(\dfrac{\|v^{\varepsilon}(x)\|_{\infty}}{k} \right)dx,
\end{equation*}
thus, 
\begin{equation*}
\int_{U}\Phi\left(\dfrac{\|v^{\varepsilon}(x)\|_{\infty}}{k} \right)dx \leq 1 \Longrightarrow \int_{U}\Phi\left(\dfrac{|v^{\varepsilon}(x)|}{k}\right)dx \leq 1,
\end{equation*}
hence, 
\begin{equation}\label{c2eq2}
\|v^{\varepsilon}\|_{L^{\Phi}(U)} \leq \|v\|_{L^{\Phi}(U; \mathcal{C}_{b}(\mathbb{R}^{d}_{y}\times \mathbb{R}^{d}_{z}))}.
\end{equation}
Thus, the trace operator $t^{\varepsilon} : v \rightarrow v^{\varepsilon}$ from $\mathcal{C}(\overline{U}; \mathcal{C}_{b}(\mathbb{R}^{d}_{y}\times \mathbb{R}^{d}_{z}))$ into $L^{\Phi}(U)$, extends by density and continuity to a unique operator from $L^{\Phi}(U; \mathcal{C}_{b}(\mathbb{R}^{d}_{y}\times \mathbb{R}^{d}_{z}))$, still denoted in the same way, which verifies (\ref{c2eq2}) for all $v \in L^{\Phi}(U; \mathcal{C}_{b}(\mathbb{R}^{d}_{y}\times \mathbb{R}^{d}_{z}))$. Referring to \cite[Section 2]{tacha3}, it can be ensured measurability for the trace of elements $v \in L^{\infty}(\mathbb{R}^{d}_{y}; \mathcal{C}_{b}(\mathbb{R}^{d}_{z}))$ and $v\in \mathcal{C}(\overline{U}; L^{\infty}(\mathbb{R}^{d}_{y}; \mathcal{C}(\mathbb{R}^{d}_{z})))$, which is of crucial importance to deal with reiterated two-scale convergence. 
\par By $M : \mathcal{C}_{per}(Y\times Z) \rightarrow \mathbb{R}$, we denote the mean value operator (or equivalently 'averaging operator') defined by 
\begin{equation}\label{c2eq3}
v \longrightarrow M(v) := \iint_{Y\times Z} v(y,z) dydz.
\end{equation} 
It is easily seen that $M$ is :
\begin{itemize}
	\item[(i)] nonnegative, i.e., $M(v) \geq 0$ for all $v \in \mathcal{C}_{per}(Y\times Z)$, $u \geq 0$;
	\item[(ii)] continuous on $\mathcal{C}_{per}(Y\times Z)$ (for the sup norm);
	\item[(iii)] such that $M(1) = 1$;
	\item[(iv)] translation invariant.
\end{itemize}
Following \cite{tacha4}, for the given $N$-function $\Phi$, by $\Xi^{\Phi}(\mathbb{R}^{d}_{y}; \mathcal{C}_{b}(\mathbb{R}^{d}_{z}))$, we denote the space 
\begin{equation}\label{c2eq4}
\begin{array}{l}
\Xi^{\Phi}(\mathbb{R}^{d}_{y}; \mathcal{C}_{b}(\mathbb{R}^{d}_{z})) := \bigg\{v \in L^{\Phi}_{loc}(\mathbb{R}^{d}_{x}; \mathcal{C}_{b}(\mathbb{R}^{d}_{z})) : \, \textup{for \, every} \, U\in \mathcal{A}(\mathbb{R}^{d}_{x}) : \\
\qquad\qquad\qquad \qquad \sup_{0<\varepsilon\leq 1} \inf \left\{k>0, \; \int_{U} \Phi\left(\dfrac{\|v(\frac{x}{\varepsilon}, \cdot)\|^{L^{\infty}}}{k}\right)dx \leq 1 \right\} < \infty \bigg\},
\end{array}
\end{equation}
which, endowed with the norm 
\begin{equation}
\begin{array}{l}
\|v\|_{\Xi^{\Phi}(\mathbb{R}^{d}_{y}; \mathcal{C}_{b}(\mathbb{R}^{d}_{z}))} := \\
\sup_{0<\varepsilon\leq 1} \inf \left\{k>0, \; \int_{B_{d}(\omega,1)} \Phi\left(\dfrac{\|v(\frac{x}{\varepsilon}, \cdot)\|_{L^{\infty}}}{k}\right)dx \leq 1 \right\},
\end{array}
\end{equation}
turns out to be a Banach space. $B_{d}(\omega,1)$, above being the unit ball of $\mathbb{R}^{d}_{x}$ centered at the origin. 
\par We denote by $\mathfrak{X}^{\Phi}_{per}(\mathbb{R}^{d}_{y}; \mathcal{C}_{b}(\mathbb{R}^{d}_{z}))$ the closure of $\mathcal{C}_{per}(Y\times Z)$ in $\Xi^{\Phi}(\mathbb{R}^{d}_{y}; \mathcal{C}_{b}(\mathbb{R}^{d}_{z}))$ and with the above notation, by $L^{\Phi}_{per}(Y\times Z)$ the space of functions in $L^{\Phi}_{loc}(\mathbb{R}^{d}_{y}\times \mathbb{R}^{d}_{z})$ which are $Y\times Z$-periodic, with the norm $\|\cdot\|_{\Phi, Y\times Z}$, (i.e., one considers the $L^{\Phi}$ norm just on the unit period). Furthermore,it is immediately seen that 
\begin{equation*}
\begin{array}{rcl}
\left|\int_{B_{d}(\omega,1)}v\left(\dfrac{x}{\varepsilon}, \dfrac{x}{\varepsilon^{2}}\right) dx \right|& \leq & \int_{B_{d}(\omega,1)} \left\|v\left(\dfrac{x}{\varepsilon}, \cdot\right)\right\|_{\infty} dx  \\
& \leq & 2 \|1\|_{\widetilde{\Phi}, B_{d}(\omega,1)} \|v\|_{\Xi^{\Phi}(\mathbb{R}^{d}_{y}; \mathcal{C}_{b}(\mathbb{R}^{d}_{z}))},
\end{array}
\end{equation*}
for every $v\in \mathcal{C}_{per}(Y\times Z)$ and for every $0< \varepsilon \leq 1$. 
\par The following results are useful to prove estimates which involve test functions on oscillating arguments (see for instance Proposition 2.7), can be found in \cite[Section 2]{tacha3}. 
\begin{lem}\label{c2eq5}
	There exists $C\in \mathbb{R}^{+}$ such that 
	\begin{equation*}
	\|v^{\varepsilon}\|_{\Phi, B_{d}(\omega,1)} \leq C \|v\|_{\Phi, Y\times Z}\;\, \textup{for\, every}\; 0< \varepsilon \leq 1 \; \textup{and} \; v \in \mathfrak{X}^{\Phi}_{per}(\mathbb{R}^{d}_{y}; \mathcal{C}_{b}(\mathbb{R}^{d}_{z})).
	\end{equation*}
\end{lem}
\begin{lem}\label{c2eq6}
	The operator $M$ defined on $\mathcal{C}_{per}(Y\times Z)$ by (\ref{c2eq3}) can be extended (with the same notation) by continuity to a unique linear and continuous operator from $\mathfrak{X}^{\Phi}_{per}(\mathbb{R}^{d}_{y}; \mathcal{C}_{b}(\mathbb{R}^{d}_{z}))$ to $\mathbb{R}$ in such a way that it results non negative and translation invariant.
\end{lem}
Finally, we recall that $\mathfrak{X}^{\Phi}_{per}(\mathbb{R}^{d}_{y}; \mathcal{C}_{b}(\mathbb{R}^{d}_{z}))$ can be endowed with another norm, considering the set  $\mathfrak{X}^{\Phi}_{per}(\mathbb{R}^{d}_{y}\times\mathbb{R}^{d}_{z})$ as the closure of $\mathcal{C}_{per}(Y\times Z)$ in $L^{\Phi}_{loc}(\mathbb{R}^{d}_{y}\times\mathbb{R}^{d}_{z})$ with the norm 
\begin{equation*}
\|v\|_{\Xi^{\Phi}} := \sup_{0<\varepsilon\leq 1} \left\|v\left(\frac{x}{\varepsilon},\frac{y}{\varepsilon}\right) \right\|_{\Phi, (B_{d}(\omega,1))^{2}}.
\end{equation*}
Via Riemann-Lebesgue Lemma it can be proven that the above norm is equivalent to $\|v\|_{L^{\Phi}(Y\times Z)}$ thus, in the sequel, we will consider this one. For completeness, we state the following result, whose proof is in the Appendix of \cite{tacha3}. It states that the latter norm is controlled by the one defined in (\ref{c2eq4}), thus together with Lemma \ref{c2eq5}, it provides the equivalence among the introduced norms in $\mathfrak{X}^{\Phi}_{per}(\mathbb{R}^{d}_{y}; \mathcal{C}_{b}(\mathbb{R}^{d}_{z}))$. 
\begin{lem}\label{c2eq7}
	We have 
	\begin{equation*}
	\mathfrak{X}^{\Phi}_{per}(\mathbb{R}^{d}_{y}; \mathcal{C}_{b}(\mathbb{R}^{d}_{z})) \subset L^{\Phi}_{per}(Y\times Z) = \mathfrak{X}^{\Phi}_{per}(\mathbb{R}^{d}_{y}\times\mathbb{R}^{d}_{z}),
	\end{equation*}
	\begin{equation*}
	\|v\|_{\Phi,Y\times Z} \leq C\,\|v\|_{\Xi^{\Phi}_{per}(\mathbb{R}^{d}_{y}; \mathcal{C}_{b}(\mathbb{R}^{d}_{z}))}, \quad \textup{for\, all} \, v \in \mathfrak{X}^{\Phi}_{per}(\mathbb{R}^{d}_{y}; \mathcal{C}_{b}(\mathbb{R}^{d}_{z})).
	\end{equation*}
\end{lem} 

\subsection{Appendix B : An abstract existence and uniqueness result}
This subsection is devoted to recall some results for existence and uniqueness of solutions of differential equations in Banach spaces. \\
Let $V$ be a real reflexive Banach space, and $H$ be a real Hilbert space such that 
\begin{equation}
V \subset H \;\, \textup{with\, continuous \, embedding}, \; V\, \textup{dense\, in}\, H, \;\; V\, \textup{is \, separable}.
\end{equation}
Let $(\cdot, \cdot)$ and $|\cdot|$ be the scalar product and the associated norm in $H$. Identifying $H$ to its dual yields $V \subset H \subset V'$, and it is convenient to denote also by $(\cdot, \cdot)$ the duality pairing between $V$ and $V'$, norms in $V$ and in $V'$ will be denoted $\|\cdot\|$. 
\par Let $0< T<\infty$ and let $\Phi\in \Delta_{2}$ which dominates the function $t\rightarrow t^{2}$. For each $t\in [0,T]$, let $A(t) : V\rightarrow V'$ be an operator satisfying the assumption (C) below : \\
(C) if $u\in V$, $v\in L^{\Phi}(0,T ; V)$, $(v_{n}) \subset L^{\Phi}(0,T ; V)$ and $\omega$ is the origin in $V$, then  
\begin{itemize}
	\item[(i)] the mapping $t\rightarrow A(t)u$ sends continuously $[0,T]$ into $V'$;
	\item[(ii)] the function $t\rightarrow A(t)u(t)$ belongs to $L^{\widetilde{\Phi}}(0,T ; V')$;
	\item[(iii)] $(A(t)\omega, w) = 0$ for all $t\in [0,T]$ and all $w\in V$; 
	\item[(iv)] there exists a constant $c>0$ such that for all $u,w \in V$, $(A(t)u - A(t)w, u-w) \geq c \Phi(\|u-w\|)$, $t\in[0,T]$;
	\item[(v)] if $v_{n}\rightarrow v$ in $L^{\Phi}(0,T ; V)$-weak, then for all $\chi \in L^{\Phi}(0,T ; V)$ one has $\lim_{n \to \infty} \int_{0}^{T} (A(t)v_{n}(t), \chi(t)) dt = \int_{0}^{T} (A(t)v(t), \chi(t)) dt$.
\end{itemize}
The following result can be found in \cite{kenne1}.
\begin{thm}\cite{kenne1}\label{apenb1}
	Assume that (C) holds. Let $f\in L^{\widetilde{\Phi}}(0,T ; V')$ and $u_{0} \in H$. There exists one and only one function $u\in \mathcal{C}([0,T]; H) \cap L^{\Phi}(0,T ; V)$ satisfying 
	\begin{equation}\label{apenb2}
	\left\{\begin{array}{l}
	u'(t) + A(t)u(t) = f(t), \quad 0< t\leq T, \\
	u(0) = u_{0}.
	\end{array}\right.
	\end{equation}
\end{thm}
Note that the first equation in (\ref{apenb2}) implies the variational formulation 
\begin{equation*}
\left(u^{\prime}(t), v\right) + \left(A(t)u(t), v\right) = \left(f(t), v\right), \quad 0< t \leq T,
\end{equation*}
for all $v\in V$.
\begin{prop}\cite{kenne1}
Assume that (C) holds. Let $f\in L^{\widetilde{\Phi}}(0,T ; V')$ and $u_{0} \in H$. Let $u\in \mathcal{C}([0,T]; H) \cap L^{\Phi}(0,T ; V)$ be the solution of (\ref{apenb2}). Then we have the following a priori estimate 
\begin{equation}\label{apenb3}
|u(t)|^{2} + c \int_{0}^{T} \Phi(\|u(t)\|) dt \leq \int_{0}^{T} \widetilde{\Phi}(\gamma\|f(t)\|) dt + |u_{0}|^{2}, \quad 0< t \leq T.
\end{equation}
\end{prop}
\begin{rem}
	According to (ii) of the assumption (C), the function $t\rightarrow f(t)- A(t)u(t)$ from $(0,T)$ to $V'$ belongs to $L^{\widetilde{\Phi}}(0,T ; V')$. Hence $u' \in L^{\widetilde{\Phi}}(0,T ; V')$ with the first equation in (\ref{apenb2}).
\end{rem}

\vspace{0.3cm}

\flushleft\textbf{Acknowledgements.} Fotso Tachago is grateful to  Department of Basic and Applied Science for Engineering of Sapienza - University of Rome for its kind hospitality, during the preparation of this work. He also acknowledges the support received by International Mathematical Union, through IMU grant 2024.  


\end{document}